\documentclass[12pt]{amsart}
\usepackage{amssymb,amsmath,amsthm,mathrsfs}
\usepackage[usenames]{color}
\usepackage{hyperref}
\usepackage[all]{xy}
\usepackage{xypic}
\usepackage{graphicx}
\usepackage{amscd}
\usepackage{enumerate}
\usepackage{ bbold }

\newcommand{\Bmu}{\mbox{$\raisebox{-0.59ex}
		{$l$}\hspace{-0.18em}\mu\hspace{-0.88em}\raisebox{-0.98ex}{\scalebox{2}
			{$\color{white}.$}}\hspace{-0.416em}\raisebox{+0.88ex}
		{$\color{white}.$}\hspace{0.46em}$}{}}

\baselineskip 18pt \textwidth 16cm  \theoremstyle{plain}

\newtheorem{theorem}{Theorem}[subsection]
\newtheorem{thm}[theorem]{Theorem}
\newtheorem{proposition}[theorem]{Proposition}
\newtheorem{corollary}[theorem]{Corollary}
\newtheorem{lemma}[theorem]{Lemma}
\newtheorem{definition}[theorem]{Definition}
\newtheorem{defn}[theorem]{Definition}
\newtheorem{notation}[theorem]{Notation}

\newtheorem{remark}[theorem]{Remark}

\newtheorem{prop}[theorem]{Proposition}

\DeclareMathOperator{\GL}{GL}

\DeclareMathOperator{\Aut}{Aut}
\DeclareMathOperator{\Hom}{Hom}
\DeclareMathOperator{\End}{End}
\DeclareMathOperator{\Ad}{Ad}

\DeclareMathOperator{\Res}{Res}

\DeclareMathOperator{\trace}{trace}
\DeclareMathOperator{\rk}{rk}
\DeclareMathOperator{\Ind}{Ind}

\DeclareMathOperator{\Spec}{Spec}

\DeclareMathOperator{\Gal}{Gal}

\DeclareMathOperator{\Fl}{Fl}

\DeclareMathOperator{\Irr}{Irr}

\newcommand{\N}{\mathbb N}
\newcommand{\bZ}{\mathbb Z }

\newcommand{\bfM}{\mathbf M }

\usepackage[usenames,dvipsnames]{xcolor}

\definecolor{dblue}{RGB}{6,69,173}
\definecolor{lblue}{RGB}{11,0,128}

\newcommand{\colorlinks}{true}

\newcommand{\linkcolor}{lblue}
\newcommand{\citecolor}{green}
\newcommand{\urlcolor}{dblue}

\newcommand{\linkbordercolor}{red}
\newcommand{\citebordercolor}{green}
\newcommand{\urlbordercolor}{cyan}

\usepackage{hyperref}
\hypersetup{colorlinks=\colorlinks,linkbordercolor=\linkbordercolor, linkcolor=\linkcolor, citecolor=\citecolor, urlcolor=\urlcolor,urlbordercolor=\urlbordercolor,citebordercolor=\citebordercolor}%

\newcommand{\hrefHid}[2]{
\hypersetup{urlbordercolor={1 1 1}}%
\hypersetup{urlcolor=black}%
\href{#1}{#2}%
\hypersetup{urlbordercolor=\urlbordercolor}%
\hypersetup{urlcolor=\urlcolor}%
}

\newcommand{\inhref}[2]{\hyperref[#1]{#2}}

\newcommand{\inhrefHid}[2]{%
\hypersetup{linkbordercolor={1 1 1}}%
\hypersetup{linkcolor=black}%
\inhref{#1}{#2}%
\hypersetup{linkbordercolor=\linkbordercolor}%
\hypersetup{linkcolor=\linkcolor}%
}

\newcommand{\defHref}[3]{\newcommand{#1}[1][#3]{\href{#2}{##1}}}
\newcommand{\defInhref}[3]{\newcommand{#1}[1][#3]{\inhref{#2}{##1}}}
\newcommand{\defHrefHid}[3]{\newcommand{#1}[1][#3]{\hrefHid{#2}{##1}}}
\newcommand{\defInhrefHid}[3]{\newcommand{#1}[1][#3]{\inhrefHid{#2}{##1}}}

\newcommand{\defHrefBoth}[3]{%
\expandafter\defHrefHid \csname #3Hid\endcsname {#1}{#2}%
\expandafter\defHref \csname #3Vis\endcsname {#1}{#2}%
}
\newcommand{\defInhrefBoth}[3]{%
  \expandafter\defInhrefHid \csname #3Hid\endcsname {#1}{#2}%
  \expandafter\defInhref \csname #3Vis\endcsname {#1}{#2}%
}

\newcommand{\defHrefBothVis}[3]{%
\defHrefBoth{#1}{#2}{#3}%
\expandafter\defHref \csname #3\endcsname {#1}{#2}%
}
\newcommand{\defInhrefBothVis}[3]{%
  \defInhrefBoth{#1}{#2}{#3}%
  \expandafter\defInhref \csname #3\endcsname {#1}{#2}%
}

\newcommand{\defHrefBothHid}[3]{%
\defHrefBoth{#1}{#2}{#3}%
\expandafter\defHrefHid \csname #3\endcsname {#1}{#2}%
}
\newcommand{\defInhrefBothHid}[3]{%
  \defInhrefBoth{#1}{#2}{#3}%
  \expandafter\defInhrefHid \csname #3\endcsname {#1}{#2}%
}

\newcommand{\Href}[3]{%
\href{#1}{#2}%
\defHrefBothVis{#1}{#2}{#3}
}

\newcommand{\wref}[2]{%
\href{http://en.wikipedia.org/wiki/#1}{#2}%
 }

\newcommand{\Wref}[3]{%
\Href{http://en.wikipedia.org/wiki/#1}{#2}{#3}%
 }

\baselineskip 18pt \textwidth 16cm  \theoremstyle{plain}

\newtheorem*{theorem*}{Theorem}
\newtheorem*{remark*}{Remark}
\newtheorem*{conjecture*}{Conjecture}

\newtheorem{introtheorem}{Theorem}

\newtheorem{introcorollary}[introtheorem]{Corollary}

\newtheorem{introconjecture}[introtheorem]{Conjecture}

\newcommand{\C}{\mathbb C}

\DeclareMathOperator{\spec}{Spec}

\DeclareMathOperator{\Rad}{Rad}

\newcommand{\ind}{\operatorname{ind}}

\newcommand{\cM}{{\mathcal M}}
\newcommand{\cK}{{\mathcal K}}
\newcommand{\cT}{{\mathcal T}}

\renewcommand{\dim}{{\operatorname{dim}}}

\renewcommand{\Hom}{{\operatorname{Hom}}}

\newcommand{\irr}{\operatorname{Irr}}
\newcommand{\Rami}[1]{{{#1}}}

\newcommand{\co}[1]{{{#1}}}

\newcommand{\NextVer}[1]{{}}

\newcommand{\bR}{\mathbf{R}}
\newcommand{\bG}{\mathbf{G}}
\newcommand{\bfG}{\mathbf{G}}
\newcommand{\bQ}{\mathbf{Q}}

\newcommand{\bH}{\mathbf{H}}
\newcommand{\bN}{\mathbf{N}}
\newcommand{\bfH}{\mathbf{H}}
\newcommand{\bfT}{\mathbf{T}}
\newcommand{\bfX}{\mathbf{X}}
\newcommand{\bfB}{\mathbf{B}}

\newcommand{\bK}{\mathbf{K}}
\newcommand{\bfK}{\mathbf{K}}

\newcommand{\bS}{\mathbf{S}}
\newcommand{\bX}{\mathbf{X}}

\newcommand{\bF}{\mathbb{F}}
\newcommand{\F}{\mathbb{F}}
\newcommand{\Z}{\mathbb{Z}}
\newcommand{\Q}{\mathbb{Q}}

\newcommand{\fL}{\mathfrak{L}}
\newcommand{\fM}{\mathfrak{M}}
\newcommand{\fK}{\mathfrak{K}}

\newcommand{\alg}{\mathrm{rd}}
\newcommand{\ralg}{\overline{\alg}}

\newcommand{\cS}{{\mathcal S}}
\renewcommand{\cR}{{\mathcal R}}
\newcommand{\cE}{{\mathcal E}}
\newcommand{\cF}{{\mathcal F}}
\newcommand{\cG}{{\mathcal G}}
\newcommand{\cI}{{\mathcal I}}
\newcommand{\calH}{{\mathcal H}}

\newcommand{\LP}{{C_{LP0}}}
\newcommand{\LPAd}{{C_{LP0}^{adj}}}
\newcommand{\LPss}{{C_{LP0}^{ss}}}
\newcommand{\LPsym}{{C_{LP}}}
\newcommand{\mon}{C_{mon0}}
\newcommand{\monSym}{C_{mon}}

\newcommand{\fX}{{\mathfrak X}}

\newcommand{\eps}{{\varepsilon}}
\newtheorem{innercustomthm}{Theorem}
\newenvironment{customthm}[1]
  {\renewcommand\theinnercustomthm{#1}\innercustomthm}
  {\endinnercustomthm}

\title{Bounds on multiplicities of symmetric pairs of finite groups}
\author{Avraham Aizenbud}
\address{Avraham Aizenbud,
Faculty of Mathematics and Computer Science, Weizmann Institute of Science, Israel. }
\email{aizenr@gmail.com}
\urladdr{\url{http://www.aizenbud.org}}

\author{Nir Avni}
\address{Nir Avni, Department of Mathematics, Northwestern University, Evanston, IL, USA.}
\email{avni.nir@gmail.com}
\urladdr{\url{http://math.northwestern.edu/\~nir}}

\newtheorem{cor}[theorem]{Corollary}

\begin{document}
\maketitle
\begin{abstract} Let $\Gamma$ be a finite group, let $\theta$ be an involution of $\Gamma$, and let $\rho$ be an irreducible complex representation of $\Gamma$. We bound $\dim \rho^{\Gamma^{\theta}}$ in terms of the smallest dimension of a faithful $\F_p$-representation of $\Gamma/\Rad_p(\Gamma)$, where $p$  is any odd prime and $\Rad_p(\Gamma)$ is the maximal normal $p$-subgroup of $\Gamma$.

This implies, in particular, that if $\bG$ is a group scheme over $\mathbb{Z}$ and $\theta$ is an involution of $\mathbf{G}$, then the multiplicity of any irreducible representation in $C^\infty \left( \mathbf{G}(\mathbb{Z}_p)/ \mathbf{G} ^{\theta}(\mathbb{Z}_p) \right)$ is bounded, uniformly in $p$.
\end{abstract}

\tableofcontents 
\section{Introduction}
 
The main result of this paper is the following:

\begin{introtheorem}[see \S \ref{ssec:ded.str} below]\label{thm:main.str} There is an increasing function $C^{fin}:\N\to \N$
such that, for any
\begin{itemize}
\item Odd prime $p$,
\item Positive integer $d$,
\item Finite group $\Gamma$,
\item Normal $p$-subgroup $N\lhd \Gamma$,
\item Embedding $\Gamma/N \hookrightarrow GL_d(\F_p)$,
\item Involution $\theta$ of $\Gamma$, 
\item Irreducible representation $\rho$ of $\Gamma$,
\end{itemize}
the space $\rho ^{\Gamma ^ \theta}$ of $\Gamma ^ \theta$-invariant vectors of $\rho$ has dimension at most $C^{fin}(d)$.
\end{introtheorem}

As a corollary, we deduce the following

\begin{introcorollary}[see \S\ref{ssec:ded.str} below]\label{cor:main} For every integer $d$, there is an integer $\Lambda$ such that, if 
\begin{itemize}
\item $p$ is an odd prime, 
\item $F$ is a purely ramified extension of $\Q_p$,
\item $\mathbf{G}$ is a connected linear algebraic group over $F$ whose reductive quotient has dimension at most $d$,
\item $K \subset \mathbf{G}(F)$ is a compact subgroup,
\item $\theta$ is an involution of $K$,
\item $\rho$ is an irreducible representation of $K$,
\end{itemize}
then $$\dim \left( \rho^{K^\theta} \right) \leq \Lambda.$$
\end{introcorollary}

\subsection {Background and motivation} Let $G$ be a group and let $X$ be a transitive $G$-space. A basic problem of representation theory is to compute the multiplicities with which irreducible representations of $G$ appear in the space of functions on $X$. This problem can be studied in several settings. In each setting, one considers a different kind of function space. For an example in the algebraic setting, if $G$ is a connected reductive algebraic group over $\mathbb{C}$ and $X$ is a spherical $G$-variety (this means that the Borel subgroup of $G$ has an open orbit in $X$), then $\mathbb{C}[X]$ is multiplicity-free as a $G$-representation.

Multiplicities for spherical $G$-varieties are of great interest in other settings. In non-algebraic settings, these multiplicities may be greater than one. However, we make the following:

\begin{introconjecture} \label{conj:bdd.mult.spherical} Let $\mathbf{G}$ be a reductive group scheme over $\mathbb{Z}$ and let $\bX$ be a $\mathbf{G}$-scheme. Assume that $\mathbf{X}(\mathbb{C})$ is a spherical $\mathbf{G}(\mathbb{C})$-space. Then there is an integer $C$ such that, if $R$ is either a finite field, a local field of large enough characteristic, or the ring of integers in such a local field, and $\rho$ is an irreducible\footnote{if $R$ is a local field, we also assume smooth and admissible} representation $\bG(R)$, then $$\dim \Hom (\rho,C^\infty(\bX(R)) ) <C.$$
\end{introconjecture}

One can study variants of this conjecture in various levels.
\begin{itemize}
\item The most basic level is when both $R$ and $\rho$ are fixed. Here the only relevant cases are when $G(R)$ is non-compact (so $R$ is a local field). This case was mostly done, see \cite{vdB87,Del,SV}. 
\item The next level  is when only $R$ is fixed and $C$ is required to be independent of $\rho$. The conjecture is non-trivial only if $R$ is infinite. The main known results in this level are for the archimedian case. See  \cite{vdB87,KO,KS16,AGM}. 
\item The last level is the full conjecture where both $R$ and $\rho$ vary. Until recently, the only known cases of this level were cases where the spherical spaces are multiplicity free (Gelfand pairs) and related situations. Although there are many known Gelfand pairs (see e.g.~%
\cite{GK,Sha,vD86,Fli91,BvD94,Nie,Yak,AGRS, AGS,OS,AG_AMOT,AG_HC,AG_RegPar,AAG,Zha,JSZ2,JSZ1,AS,AGJ,AG_M1J,SZ,Aiz,CS,Car,Rub}, and the reference therein), general spherical spaces are not multiplicity free. 
\item Conjecture \ref{conj:bdd.mult.spherical} when $R$ ranges over all finite fields was recently proved in \cite{AA_bnd,She}.
\end{itemize}

Corollary \ref{cor:main} is a first step in generalizing \cite{AA_bnd,She} from finite fields to rings of integers in local fields. More precisely, we work with rings of integers of purely unramified extensions of $\mathbb{Q}_p$ and restrict the class of spherical spaces to the class of symmetric spaces, which are spaces of the form $\mathbf{X}=\mathbf{G} / \mathbf{G} ^ \theta$, where $\theta$ is an involution. If $O$ is the ring of integers in a non-archimedean local field and $\mathbf{X}=\mathbf{G} / \mathbf{G} ^ \theta$ is a symmetric space, then $\mathbf{X}(O)$ is a union of a bounded number of sets of the form $\mathbf{G}(O)/\mathbf{G}(O)^{\theta_i}$, for some involutions $\theta_i$ of $\mathbf{G}$, and
\[
\dim \Hom \left( \rho,C^\infty(\mathbf{G}(O)/\mathbf{G}(O)^{\theta_i})\right) = \dim \rho ^{\mathbf{G}(O)^{\theta_i}},
\]
for which Corollary \ref{cor:main} applies.


\begin{remark} Originally, we were interested only in the special case of Corollary \ref{cor:main} in which $K=\mathbf{G}(O_F)$. However, since our argument is inductive, it turns out to be easier to prove a more general claim.
\end{remark}

\subsection{The Larsen--Pink Theorem} A central ingredient in the proof of Theorem \ref{thm:main.str} is a theorem of \cite{LP} roughly stating that finite subgroups of $\GL_d(\F_p)$ are close to groups of $\F_p$-points of connected algebraic subgroups of  $\GL_d$. We use the Larsen--Pink theorem in several ways:
\begin{itemize}
\item The Larsen--Pink theorem attaches an algebraic group of $\GL_n$ to finite subgroups of $\GL_n(\mathbb{F}_p)$, and we prove Theorem \ref{thm:main.str} by induction on the dimension on this algebraic group. In particular, the Larsen--Pink theorem implies that the lengths of decreasing chains of perfect subgroups of $\GL_d(\F_p)$ are bounded when we vary $p$. 
\item It allows us to reduce statements about finite groups with no normal $p$-subgroups to finite groups of Lie type. We use this to prove the main theorem for groups with trivial $p$-radical (See \S\ref{sec:triv.prad}) and to get bounds on various cohomology groups in \S\ref{sec:H1} and \S\ref{sec:H2}.
\end{itemize}
We discuss the Larsen--Pink theorem and its applications in \S\ref{sec:LP}.

\subsection{Sketch of the proof of Theorem \ref{thm:main.str}}

We first analyze the case of groups with odd order. The analysis is based on the simple observation that every element in such a group has a unique square root. We prove a Gelfand property (i.e., multiplicity one property) for symmetric pairs of such groups. In addition, we prove a necessary condition (related to conjectures of Lapid and Prasad \cite{Gla18,Pra}) for a representation of a group $G$ of odd order to be distinguished with respect to a symmetric subgroup of $G$. Finally, we show that the first cohomology of $S_2$ with coefficients in groups of odd order vanishes. We treat this case in \S\ref{sec:odd}.\\

Next, we analyze the case of a group with a trivial $p$-radical. Here, we prove a twisted version of the main theorem. Using the Larsen-Pink theorem, we reduce this case to the case of finite groups of Lie type, where we apply a similar reasoning as in \cite{AA_bnd,She}. We treat this case in \S\ref{sec:triv.prad} \Rami{and Apendix \ref{sec:Shai}}.\\

For the general case, we introduce the following invariant of a group $\Gamma$: $\ralg_p(\Gamma)$ is the smallest possible dimension of a connected reductive group $\mathbf{G}$ such that $\Gamma / \Rad_p(\Gamma) \subseteq \mathbf{G}(\mathbb{F}_p)$. Since, in the notations of Theorem \ref{thm:main.str}, $\ralg_p(\Gamma) \leq d^2$, it is enough to bound $\dim \rho ^{\Gamma ^ \theta}$ in terms of $\ralg_p(\Gamma)$. This is done by induction on $\ralg_p(\Gamma)$. In the rest of the section, we describe the induction step.\\

Clifford theory implies that there is a group $\Delta$ satisfying $\Rad_p(\Gamma) < \Delta<\Gamma$ and an irreducible representation $\sigma$ of $\Delta$ such that
\begin{itemize}
\item $\rho=Ind_\Delta^\Gamma(\sigma)$.
\item $\sigma|_{\Rad_p(\Gamma)} $ is isotypic.
\end{itemize}
By Mackey's formula, the multiplicity $\dim \rho^{\Gamma^\theta}$ is a sum of multiplicities of $\sigma$ in various transitive $\Delta$-sets.  A-priory, the number of transitive $\Delta$-sets that might contribute to $\dim \rho^{\Gamma^\theta}$ is $|\Gamma^\theta   \backslash \Gamma/\Delta|$, which is unbounded. We use the Lapid--Prasad criterion to bound the number of subgroups of $\Delta$ whose contribution is non-zero by $|H^1(S_2,\Delta)|$, which we can bound. 

In order to bound the individual contribution of a transitive $\Delta$-set, we analyze two possibilities:
\begin{itemize}
\item $\Rad_p(\Delta)=\Rad_p(\Gamma)$.\\ 
In this case, the bound on $H^2(\Delta / \Rad_p(\Delta),\Bmu_{p^{\infty}})$ implies that, for large $p$, the representation $\sigma$ is a tensor product of a representation $\sigma_1$ that is trivial on $\Rad_p(\Delta)$ and a representation $\sigma_2$ that is irreducible when restricted to $\Rad_p(\Delta)$. The multiplicity of $\sigma_2$ is at most one, since $\Rad_p(\Delta)$ has odd order. The bound on the multiplicity $\sigma_1$ follows from the analysis of the case with trivial $p$-radical mentioned above. At this point of the argument, we need to bound twisted multiplicities of representations of $\Delta / \Rad_p(\Delta)$, rather than usual multiplicities. The reason is that the one-dimensional multiplicity space obtained for $\Rad_p(\Delta)$ manifests itself as a twist here.
\item $\Rad_p(\Delta)\neq \Rad_p(\Gamma)$.\\
In this case, the Larsen--Pink Theorem implies that there is subgroup of bounded index $\Delta^\circ\lhd \Delta$ such that $\ralg(\Delta^\circ)<\ralg(\Gamma)$. We deduce the required bound from the induction assumption.
\end{itemize}

\subsection{Complication related to the action of ${S_2}$} The sketch above overlooks one technical point. Namely, while the Larsen--Pink theorem was proved for groups, we need it for symmetric pairs or, equivalently, in the $S_2$-equivariant setting. One way around this difficulty is to embed $\Gamma$ into $\Gamma\times \Gamma$ using the graph of $\theta$. Under this embedding, $\theta$ becomes the flip $(x,y) \mapsto (y,x)$, which clearly extends to the ambient algebraic group. This way is implemented in Lemma \ref{lem:doubling} below. The drawback of this method is that it doubles the dimension of the ambient algebraic group, so it is not suitable for induction. So, in some parts of the argument, we use a different method: We use an iterative procedure, based on the Larsen-Pink Theorem that allows us to replace (without increasing the value of $\ralg_p$) a subgroup of bounded index with a smaller $S_2$-invariant subgroup, also of bounded index. We implement this procedure in Lemma \ref{lem:norm.sub} below. This procedure is very costly in terms of the bounds on the  indexes, and it is one of the main reasons why our bound on the multiplicities is very large.

\subsection{limitation of our result} 
\begin{itemize}
\item Our bounds on the multiplicities are given in terms of an embedding into a group of $\mathbb{F}_p$-points rather than a group of $\overline{\mathbb{F}_p}$-points. Therefore, we do not bound the multiplicity of symmetric pairs with $G=\mathbf{G}(O_F)$ when $F$ ranges over all extensions of a given local non-archimedean field (of course, if the degree $[O_F/\mathfrak{m}_F : \bF_p]$ is fixed, we do get uniform bounds). The reason is that, unlike $\mathbf{G}(\F_p)$, the group $\mathbf{G}( \overline{\F_p})$ has decreasing chains of perfect subgroups of arbitrary length. For this reason, we do not conjecture that Theorem \ref{thm:main.str} holds if we replace $\F_p$ with $\overline{\F_p}$. 

\item The bounds on the multiplicities we obtain are extremely large. We did not try to optimize the bounds since our argument cannot prove any reasonable bounds.

\item In the general case, we only bound usual multiplicities and not twisted ones. The reason is that our analysis of odd order groups does not work well in the twisted case. We do not expect any problems with the twisted Gelfand property, and also think that it will be easy to obtain a criterion for twisted distinction. However, this criterion will be different from the untwisted, so the number of symmetric subgroups of $\Delta$ contributing to the multiplicity will not be the size of any homology but rather some other number that we do not know how to bound. It would be interesting to resolve the twisted case, especially since we use bounds on twisted multiplicities in the case of trivial $p$-radical in order to bound the usual multiplicities in the general case.
\end{itemize}
\subsection{Structure of the paper}

In \S\ref{sec:not} we fix notations and formulate our main result. See Theorem \ref{thm:main}. In particular we introduce an invariant $\ralg_p$ to measure ``dimension" of a finite group. See Definition \ref{dfn:rd}.

In \S\ref{sec:prel} we recall some group theoretic facts.

In \S\ref{sec:LP} we quote the Larsen--Pink Theorem and deduce two corollaries that we will use in the paper:
Corollary \ref{cor:LP.reductive.with.involution} and Corollary \ref{cor:LP.ind.sym}.

In \S\ref{sec:odd} we treat the case of groups of odd order. In this case, we prove  a stronger form of the main result along with some other results for this special case. See Lemma \ref{lem:H1odd}  and Corollary \ref{cor:gel.Lap.Pra}.

In \S\ref{sec:H1} we bound the size of the first cohomology group of $S_2$ with coefficients in a finite group $\Gamma$ in terms of $\ralg_p(\Gamma)$. See Corollary \ref{cor:h1.her}.

In \S\ref{sec:H2} we prove a vanishing result for $H^2(\Gamma,\mathbb{Z} / p^n)$, where $\Gamma$ is a finite group, assuming $p$ is large enough with respect to $\ralg_p(\Gamma)$. See  Proposition \ref{prop:H2}.

In \S\ref{sec:triv.prad} we prove a twisted version of the main result for the case of groups with trivial $p$-radical.  See  Corollary \ref{cor:nr}.

In \S\ref{sec:Clif} we recall some basic results of Clifford theory which are needed in our proof.

In \S\ref{sec:pf.main} we prove our main result, Theorem \ref{thm:main}. In \S\S \ref{ssec:ded.str} we deduce Theorem \ref{thm:main.str} and Corollary \ref{cor:main} from Theorem \ref{thm:main}.

In Appendix \ref{sec:Shai} we prove a twisted version of the main result for finite groups of Lie type. The argument is an adaptation (to the twisted case) of \cite{She}.

In Appendix \ref{app:fam} we construct a family of symmetric pairs of reductive groups that includes all symmetric pairs of reductive groups of a given dimension over all finite fields, see Lemma \ref{lem:fam.of.sym}. We use this construction in \S\ref{sec:LP} in order to express the bounds given by the Larsen--Pink theorem in terms of $\ralg_p$.

\subsection{Acknowledgements} We thank Uri Bader, Shachar Carmeli, Yoav Segev, and Shai Shechter for helpful discussions. 

\Rami{
A.A. was partially supported by ISF grant 249/17 and a Minerva
foundation grant. N.A. was partially supported by NSF grant DMS-1902041. We were both partially supported by BSF grant 2018201.}

\section{Conventions, notations, and reformulation of the main result}\label{sec:not}
\subsection{Conventions}
\begin{itemize}
\item By a finite symmetric pair, we will mean a pair $(\Gamma,\theta)$, where $\Gamma$ is a finite group and $\theta$ is a (possibly trivial) involution of $\Gamma$. For a symmetric pair $(\Gamma,\theta)$, we get a symmetric subgroup $\Gamma^\theta \subset \Gamma$, a symmetric space $\Gamma/\Gamma^\theta$, and an action of $S_2$ on $\Gamma$. 
\item For a group $G$, we denote the derived subgroup of $G$ by $G'$ and the center of $G$ by $Z(G)$. If $\mathbf{G}$ is an algebraic group, we denote the connected component of identity in $\mathbf{G}$ by $\mathbf{G} ^\circ$. 
\item All schemes considered in this paper are assumed to be of finite type over Noetherian base schemes.
\item By a simple algebraic group we mean a connected algebraic group whose Lie algebra is simple. 
\item Throughout the paper, we will formulate and prove several lemmas that assert the existence of increasing functions $\N\to \N$ satisfying certain conditions. Each of those lemmas will give the corresponding function a distinct notation. It is implied that, after each such lemma, we fix such a function and use that notation to refer to it. The choices of such functions are not unique, but the only effect of a different choice is different bounds. Since we just claim the existence of bounds, this is irrelevant to us.
\item We will usually use capital boldface letters to denote varieties, capital calligraphic letters to denote schemes, and capital gothic letters to denote sheaves.
\end{itemize}
\subsection{Notations}

We will use the following invariants of a finite group.
\begin{defn}\label{dfn:rd}
Let $\Gamma$  be a finite group and $p$ be a prime.
\begin{enumerate}
\item Define the $p$-reductivity dimension $\alg_p(\Gamma)$ of $\Gamma$ to be the minimal $n$ such that there exist a connected $n$-dimensional reductive algebraic group $\bG$ and an embedding $\Gamma \hookrightarrow \bG(\F_p)$.
\item Define the reduced $p$-reductivity dimension $\ralg_p(\Gamma)$ by  $$\ralg_p(\Gamma):=\alg_p(\Gamma/\Rad_p(\Gamma)),$$ where $\Rad_p(\Gamma)$ is the maximal normal $p$-subgroup of $\Gamma.$
\end{enumerate}
\end{defn}

\begin{defn} Let $(\Gamma,\theta)$  be a finite symmetric pair and let $p$ be a prime.
\begin{enumerate}
\item Define the $p$-reductivity dimension $\alg_p(\Gamma,\theta)$ of $(\Gamma,\theta)$ to be the minimal $n$ such that there exist an $n$-dimensional reductive algebraic group $\bG$, an involution $t$ of $\bG$, and an embedding $i:\Gamma \to \bG(\F_p)$ such that $i(\theta(\gamma))=t(i(\gamma))$ for all $\gamma\in \Gamma$.
\item Define the reduced $p$-reductivity dimension $\ralg_p(\Gamma,\theta)$ by  $$\ralg_p(\Gamma,\theta):=\alg_p(\Gamma/\Rad_p(\Gamma),\bar \theta),$$ where $\bar \theta$ is the involution of $\Gamma/\Rad_p(\Gamma)$ induced by $\theta$.
\end{enumerate}
\end{defn}

\begin{lemma} \label{lem:doubling} For any symmetric pair $(\Gamma,\theta)$ and every $p$, $\alg_p(\Gamma,\theta)\leq 2 \alg_p(\Gamma)$.
\end{lemma}
\begin{proof} Let $i:\Gamma \hookrightarrow \mathbf{G}(\mathbb{F}_p)$ with $\mathbf{G}$ reductive and $\dim \, \mathbf{G}=\alg_p(\Gamma)$. Let $\mathbf{H}=\mathbf{G} \times \mathbf{G}$, let $t:\mathbf{H} \rightarrow \mathbf{H}$ be the flip $\theta(x,y)=(y,x)$, and let $j:\Gamma \hookrightarrow \mathbf{H}(\mathbb{F}_p)$ be $j(\gamma) =(i(\gamma),i(\theta(\gamma)))$. The triple $(\mathbf{H},t,j)$ gives an equivariant embedding as required.
\end{proof}

Next, we introduce some notations relating to multiplicities.
\begin{notation} \label{nota:mu.nu}
Let $\Gamma$  be a finite group.
\begin{itemize}
\item We denote the set of (isomorphism classes of) complex irreducible representations of $\Gamma$ by $\irr (\Gamma)$.
\item We denote the set of (one dimensional) characters of $\Gamma$  by $\widehat{\Gamma}$.
\item Suppose that $\theta$ is an involution of $\Gamma$. Denote 
$$\nu(\Gamma,\theta)=\max_{\rho\in \irr(\Gamma)}  \dim \rho^{\Gamma^\theta}, $$
$$\nu(\Gamma)=\max_{\text{$\theta$ involution of $\Gamma$}} \nu(\Gamma,\theta),$$
$$\nu'_p(\Gamma)=\max_{\text{$\theta$ involution of $\Gamma$}}\max_{\substack{\rho\in \irr(\Gamma) \text{ such that}\\ \text{$\rho|_{\Rad_p(\Gamma)}$ is isotypic}}}  \dim \rho^{\Gamma^\theta},$$
$$\mu(\Gamma,\theta)=\max_{\rho\in \irr(\Gamma),\,\, \chi \in \widehat{\Gamma^\theta}}  \dim \rho^{\Gamma^\theta,\chi},$$
and
$$\mu(\Gamma)=\max_{\text{$\theta$ involution of $\Gamma$}} \mu(\Gamma,\theta).$$
\end{itemize}

\end{notation}

\subsection{Reformulation of the main theorem}

Using the notations above, Theorem \ref{thm:main.str} has the following reformulation:
\begin{theorem}[main]\label{thm:main}
There is an increasing function $C:\N\to \N$ such that, for every prime $p>2$ and every finite group $\Gamma$,
$$\nu(\Gamma)<C(\ralg_p(\Gamma)).$$
\end{theorem}
This theorem appears to be slightly weaker then Theorem \ref{thm:main.str}. We will deduce Theorem \ref{thm:main.str} from it in \S\ref{ssec:ded.str}.

\section{Preliminaries on finite groups \Rami{and algebraic groups}}\label{sec:prel}
In this section, we collect several properties of finite and algebraic groups. 

\subsection{Finite groups of Lie type}
The following theorem is well known:

\begin{theorem}\label{thm:simp.con} For any finite field $F$ which is not one of $\bF_2,\bF_3,\bF_4,\bF_8,\bF_9$ and for any connected, simply-connected, semi-simple algebraic group $\bG$ defined over $F$, the following hold
\begin{enumerate}
\item \label{thm:simp.con:0} $\mathbf{G}(F)$ is generated by its unipotents.
\item \label{thm:simp.con:1} $H^2(\bG( F),A)=1$, for every trivial $\mathbf{G}(F)$-module $A$.
\item \label{thm:simp.con:2} $\bG(F)$ is perfect.
\item \label{thm:simp.con:2.5} $Z(\mathbf{G} (F))=Z(\mathbf{G})(F)$. 
\item \label{thm:simp.con:3} $|Z(\bG( F))|\leq 2^{\dim \bG}$.
\end{enumerate}
\end{theorem}

\begin{remark} Claim \eqref{thm:simp.con:0} follows from \cite[Theorem 12.4]{Ste68}, so Claim \eqref{thm:simp.con:1} follows from \cite[Remark 12.8(b)]{Ste68}. 

Claim \eqref{thm:simp.con:2} also follows, since every unipotent is contained in the unipotent radical of some $F$-rational Borel, every unipotent radical is generated by root subgroups, and every element in a root subgroup is a commutator. 

Claim \eqref{thm:simp.con:2.5} follows from \cite[Theorem 1.5.6 (i)]{Mar} and Claim \eqref{thm:simp.con:0}.


Since $|Z(\bG(F))| \leq |Z(\bG(\overline{F}))|$ and the latter is the determinant of the Cartan matrix of the Dynkin diagram of $\bG$, Claim \eqref{thm:simp.con:3} follows.
\end{remark}

\begin{corollary} \label{cor:univ.cover} Let $F$ be a finite field of characteristic greater than $3$, let $\mathbf{G}$ be a connected reductive group over $F$, and let $\phi:\widetilde{\mathbf{G} '} \rightarrow \mathbf{G}'$ be the universal cover of the derived subgroup $\mathbf{G}'$. Then $\phi \left( \widetilde{\mathbf{G}'} (F)\right) = \mathbf{G}(F)'$.
\end{corollary} 

\begin{proof} The inclusion $\phi \left( \widetilde{\mathbf{G}'} (F)\right) \subseteq \mathbf{G}(F)'$ follows from \ref{thm:simp.con}\eqref{thm:simp.con:2}. For the other direction, it is enough to show that a commutator of two elements of $\mathbf{G}(F)$ belongs to $\phi \left( \widetilde{\mathbf{G}'} (F)\right)$. Let $g_1,g_2\in \mathbf{G}(F)$. Choose $z_1,z_2\in Z\left( \mathbf{G}(\overline{F})\right)$ such that $g_1z_1,g_2z_2\in \mathbf{G}'(\overline{F})$, and choose elements $h_1,h_2\in \widetilde{\mathbf{G}'}(\overline{F})$ such that $\phi(h_1)=g_1z_1,\phi(h_2)=g_2z_2$. Since $\phi ^{-1} (Z(\mathbf{G}'))=Z \left( \widetilde{\mathbf{G}'} \right)$, the element $[h_1,h_2]\in \widetilde{\mathbf{G}'}(\overline{F})$ is independent of the choices of $z_i$ and $h_i$, and hence is fixed by $\Gal(\overline{F}/F)$. Therefore, $[h_1,h_2]\in \widetilde{\mathbf{G}'}(F)$. Hence, $[g_1,g_2]=\phi([h_1,h_2])\in \phi \left( \widetilde{\mathbf{G}'} (F)\right)$.
\end{proof} 

Corollary \ref{cor:univ.cover} and Theorem \ref{thm:simp.con}\eqref{thm:simp.con:2} imply

\begin{corollary} \label{cor:perfect.prime} Let $F$ be a finite field of characteristic greater than $3$ and let $\mathbf{G}$ be a connected reductive group over $F$. Then $\mathbf{G}(F)'$ is perfect.
\end{corollary} 

\begin{lemma}\label{lem:coker}    Let $\phi:\tilde \bG\to \bG$ be an isogeny of algebraic groups defined over a finite field $\mathbb{F}_q$ and let $\bK$  be the kernel of $\phi$.
Then 
$$
\left[ \mathbf{G}(\mathbb{F}_q) : \phi \left( \widetilde{\mathbf{G}}(\mathbb{F}_q) \right) \right] \leq |\mathbf{K}(\overline{\mathbb{F}}_q)|
$$
\end{lemma}
\begin{proof}
From the long exact sequence of Galois cohomologies
\[
\mathbf{K}( \mathbb{F} _q) \rightarrow \widetilde{\mathbf{G}}( \mathbb{F} _q ) \rightarrow \mathbf{G}( \mathbb{F} _q ) \rightarrow H^1(Gal(\overline {\mathbb{F}}_q/\mathbb{F}_q),  \mathbf{K}(\overline{\mathbb{F}}_q)) \rightarrow H^1(Gal(\overline {\mathbb{F}}_q/\mathbb{F}_q), \widetilde{\mathbf{G}}(\overline{\mathbb{F}}_q)),
\]
we get
\begin{align*}
\left[ \mathbf{G}(\mathbb{F}_q) : \phi \left( \widetilde{\mathbf{G}}(\mathbb{F}_q) \right) \right]
&=
|Ker(H^1(Gal(\overline {\mathbb{F}}_q/\mathbb{F}_q),  K(\overline{\mathbb{F}}_q))\to H^1(Gal(\overline {\mathbb{F}}_q/\mathbb{F}_q), \widetilde{\mathbf{G}}(\mathbb{F}_q) ))|
\\&
\leq 
|H^1(Gal(\overline {\mathbb{F}}_q/\mathbb{F}_q),  K(\overline{\mathbb{F}}_q))|
\leq 
|K(\overline{\mathbb{F}}_q)|
\end{align*}
\end{proof}

\begin{lemma}\label{lem:der} 
For every connected reductive group $\bG$ defined over a finite field $F$ \Rami{of characteristic larger than 3}, we have 
$$[\bG'(F):\bG(F)']<2^{\dim(\bG)}.$$
\end{lemma}
\begin{proof}
We can assume that $\bG$ is semisimple. Let $\phi :\widetilde{\mathbf{G}} \rightarrow \mathbf{G}$ be the universal cover. By Theorem \ref{thm:simp.con}\eqref{thm:simp.con:3} and Lemma \ref{lem:coker}, $| coker \phi | \leq 2^{\dim \mathbf{G}}$. The result follows from Corollary \ref{cor:univ.cover}.
\end{proof}

\begin{lemma} \label{lem:Goursat}
Let $\Gamma_i$ be simple non-abelian groups and let $\Gamma:=\prod_{i=1}^n \Gamma_i$. Any normal subgroup $\Delta$ of $\Gamma$ is of the form $\Delta=\prod_{i\in I} \Gamma_i$ for some index set $I\subset \{1,\dots,n\}$. The same holds when $ \Gamma_i$ are simple adjoint algebraic groups and $ \Delta $ is a normal algebraic subgroup of $ \Gamma $.
\end{lemma} 

\begin{proof} Assume that $ \Gamma _i$ are simple non-abelian groups, and identify $ \Gamma _i$ as a subgroup of $ \Gamma $. Let $\pi_i: \Gamma \rightarrow \Gamma_i$ be the projection. If $ \pi _i( \Delta ) \neq 1$, then $[\Delta , \Gamma _i]$ is a non-trivial normal subgroup of $ \Gamma _i$, so $ \Gamma _i = [ \Delta , \Gamma _i] \subset \Delta $. Thus, the lemma holds with $I= \left\{ i \mid \pi_i( \Delta ) \neq 1 \right\}$.

The proof in the case $ \Gamma _i$ are algebraic is similar.
\end{proof}

\begin{cor} \label{cor:aut.perm} Let $\Gamma_i$ and $\Gamma$  be as in Lemma \ref{lem:Goursat}, and let $\theta:\Gamma\to\Gamma$ be an automorphism. Then there is a permutation $\sigma\in S_n$ such that $\theta(\Gamma_i)=\Gamma_{\sigma(i)}$.
\end{cor}

\begin{cor} \label{cor:ext}
If $p>3$ is a prime number and $\mathbf{G}$ is a connected semisimple adjoint group defined over $\mathbb{F}_p$, then any automorphism of $\mathbf{G}(\mathbb{F}_p)'$ extends to an algebraic automorphism of $\mathbf{G}$.
\end{cor}

\Rami{
For the proof we will need the following:
\begin{thm}[{easy direction of the classification of finite simple groups, cf. \cite[\S1]{GLS} and \cite[3.2]{Ste60}}] \label{thm:class}  For any two finite fields $F_1,F_2$ of characteristic greater than 3 and any two absolutely simple and adjoint algebraic groups $\bG,\bH$ defined over  $F_1,F_2$ respectively,
\begin{enumerate}
\item\label{thm:class:1} $\bG(F_1)'$ is simple.
\item\label{thm:class:2} If $\bG(F_1)'\simeq \bH(F_2)'$ then $F_1\simeq F_2$ and $\bG\simeq\bH$.
\item\label{thm:class:3} Any isomorphism $\bG(F_1)'\to \bH(F_2)'$ is the composition of $\phi:\bG(F_1)' \rightarrow \bG^\phi (F_2)'$ and $\psi: \bG^\phi (F_2)' \rightarrow \bH(F_2)$, where $\phi:F_1 \rightarrow F_2$ is a field isomorphism, $\bG^\phi=\bG \times_{\Spec(F_1)} \Spec(F_2)$, and $\psi$ is the restriction of an isomorphism $\bG^\phi \rightarrow \bH$.
\end{enumerate}
\end{thm}
}

\begin{proof}[Proof of Corollary \ref{cor:ext}]
Write $\bG=\prod \bG_i$ where $\bG_i$ are simple (not necessarily absolutely simple) adjoint groups defined over $\mathbb{F}_p$. For each $i$, there is a finite field $\mathbb{F}_{q_i}$ and an absolutely simple adjoint group $\mathbf{S}_i$ defined over $\mathbb{F}_{q_i}$ such that $\bG_i\cong (\bS_i)_{\F_{q_i}/\F_p}$. We have $\bG(\F_q)'=\prod \bS_i(\F_{q_i})'$. By Theorem \ref{thm:class}\eqref{thm:class:1}, the groups $\bS_i(\F_{q_i})'$ are simple. By Corollary \ref{cor:aut.perm}, there is a permutation $\sigma$ such that $\theta(\bS_i(\F_{q_i})')=\bS_{\sigma(i)}(\F_{q_\sigma(i)})'$. The assertion now follows from Theorem \ref{thm:class}\eqref{thm:class:3}.
\end{proof}

\subsection{A versal family of reductive groups}
In order to prove uniform results for all reductive groups of a bounded dimension over an arbitrary finite field, we will use the following lemma.
\begin{lemma} \label{lem:fam.of.sym} For any integer $n>0$, there exist a scheme $\cS_n$ of finite type, a smooth group scheme $\Phi_n:\cR_n\to \cS_n$, and an involution $\tau_n:\cR_n\to \cR_n$ over $\cS_n$ such that the following hold: 
\begin{enumerate}
\item \label{lem:fam.of.sym:1} For every finite field $F$ and every $s\in \mathcal{S}_n(F)$, the group $(\mathcal{R}_n)_s$ is connected and reductive.
\item \label{lem:fam.of.sym:2} For every connected and reductive group $\mathbf{G}$ of dimension at most $n$ over a finite field $F$ and for any involution $t$ of $\mathbf{G}$, there is $s \in \mathcal{S}_n(F)$ with $$(\mathbf{G},t)\simeq ((\mathcal{R}_n)_s,(\tau_n)_s).$$
\item \label{lem:fam.of.sym:3} For any root datum $\fX$, there is a subscheme  $\cS^\fX\subset \cS_n$ such that, for any geometric point $x$ of $\cS_n$, the (absolute) root datum of  $(\cR_n)_x$ is $\fX$ if and only if $x$ factors through $\cS^\fX$. Moreover, $\mathcal{S} ^\mathfrak{X}$ is a union of connected components of $\mathcal{S}_n$.
\end{enumerate}
\end{lemma}
We prove this lemma is in Appendix \ref{app:fam}. The proof does not work for infinite fields, but we do have the following:
\begin{lemma} \label{cor:faithful.rep.local} There is a function $C^{lin}:\mathbb{N} \rightarrow \mathbb{N}$ such that any reductive group $\mathbf{G}$ over an arbitrary field $F$ has a faithful $F$-representation of dimension at most $C^{lin}(\dim\, \mathbf{G})$.
\end{lemma}


\section{A theorem of Larsen--Pink and its applications} \label{sec:LP}

A theorem of Larsen and Pink is central to our proof. In this section, we quote the theorem and extract two corollaries (Corollaries \ref{cor:LP.reductive.with.involution} and \ref{cor:LP.ind.sym}) from it.


\begin{definition} Let $\mathcal{S}$ be a scheme and let $f:\mathcal{G} \rightarrow \mathcal{S}$ be a group scheme over $\mathbf{S}$. \begin{enumerate} 
\item A {\em family of subgroups} is a pair consisting of a map $\pi:\mathcal{T} \rightarrow \mathcal{S}$ and a \co{$\mathcal{T}$-}subgroup scheme $\mathcal{H} \subset \mathcal{G} \times_{\mathcal{S}} \mathcal{T}$. In this case, we write $\mathcal{H} \Subset_\pi \mathcal{G}$.
\item Suppose that $\mathcal{H} \Subset_\pi \mathcal{G}$, that $k$ is a field, that $s\in \mathcal{S}(k)$, and that $\Gamma \subset \mathcal{G}_s(k)$ is a subgroup. We say that $\Gamma$ {\em $k$-evades} $\mathcal{H}$ if, for every $t\in \pi ^{-1} (s)(k)$, we have $\Gamma \not\subset \mathcal{H}_{t}(k)$.
\end{enumerate} 
\end{definition} 

\begin{definition} Suppose $\Gamma, \Delta$ are subgroups of some group. We say that $\Gamma$ is big in $\Delta$ if $[\Delta,\Delta] \subset \Gamma \subset \Delta$.
\end{definition} 

The following is a restatement of \cite[Theorem 0.5]{LP}:

\begin{theorem} \label{thm:LP} Let $\mathcal{S}$ be a scheme and let $\mathcal{G} \rightarrow \mathcal{S}$ be a group scheme over $\mathcal{S}$ such that every geometric fiber is connected, simple, and adjoint. There is a family of subgroups $\mathcal{H} \Subset_\pi \mathcal{G}$ such that, for every prime $p$, every $\overline{\mathbb{F}_p}$-point $s:\Spec(\overline{\mathbb{F}_p}) \rightarrow \mathcal{S}$, and every $\Gamma \subset \mathcal{G}_s(\overline{\mathbb{F}_p})$, if $\Gamma$ $\overline{\mathbb{F}_p}$-evades $\mathcal{H}$, then there is a Frobenius map $\Phi : \mathcal{G}_s \rightarrow \mathcal{G}_s$ such that $\Gamma$ is big in $\mathcal{G}_s(\overline{\mathbb{F}_p})^{\Phi}$.
\end{theorem} 

Recall that, for an algebraic group $\bfG$ defined over $\overline{\bF_p}$, a Frobenius map of $\mathbf{G}$ is an automorphism $\Phi : \mathbf{G} \rightarrow \mathbf{G}$ for which some positive power $\Phi ^n$ coincides with some a standard Frobenius.

\begin{corollary} \label{cor:LP.non.ac} Let $\mathcal{S}$ be a scheme and let $\mathcal{G} \rightarrow \mathcal{S}$ be a group scheme whose geometric fibers are connected, simple, and adjoint. There is a family of subgroups $\mathcal{K} \Subset_\tau \mathcal{G}$ such that, for every prime power $q$, every $s\in \mathcal{S}(\mathbb{F}_q)$, and every $\Gamma \subset \mathcal{G}_s(\mathbb{F}_q)$, if $\Gamma$ $\mathbb{F}_q$-evades $\mathcal{K}$, then there is a Frobenius map $\Phi:\mathcal{G}_s(\overline{\mathbb{F}_q}) \rightarrow \mathcal{G}_s(\overline{\mathbb{F}_q})$ such that $\Gamma$ is big in $\mathcal{G}_s(\overline{\mathbb{F}_q})^{\Phi}$.
\end{corollary} 

For the proof of corollary \ref{cor:LP.non.ac}, we use the following preparations:

\begin{defn} Let $\cS$ be a scheme.	Fix a closed embedding $(\GL_n)_\mathbf{S} \hookrightarrow \mathbb{A}_{\mathcal{S}}^N$. 
\begin{enumerate}
\item We say that a regular function $f: (\GL_n)_{\cS} \rightarrow \mathbb{A}^1$ has degree at most $\delta$ if it is the restriction of a polynomial of degree at most $\delta$ on $\mathbb{A}^N_{\cS}$. 
\item We say that the degree of $f: (\GL_n)_{\cS} \rightarrow \mathbb{A}^1$ is $\delta$ if its degree is at most $\delta$ and not at most $\delta-1$. 	
\item  Define the complexity of an $\mathcal{S}$-subgroup scheme $\mathcal{L} \subset (\GL_n)_{\cS}$ to be the minimal $m$ such that the polynomials of degree at most $m$ in the ideal $I(\mathcal{L})$ generate $I(\mathcal{L})$. 
\end{enumerate}
\end{defn}

\begin{lemma}	\label{lem:ultra}
For any two integers $n$ and $A$, there is an integer $B$ such that, for any field $F$, if $\mathbf{L} \subset (\GL_n)_{F}$ is an algebraic subgroup of complexity at most $A$, then $\mathbf{L}^\circ$ is of complexity at most $B$
 \end{lemma}

\begin{proof}
For any integer $B$, the statement ``the complexity of the connected component of an algebraic group $\mathbf{L}$ is at most $B$" is a first order statement on the coefficients of the polynomials defining $\mathbf{L}\subset (GL_n)_{F}$. 

The result follows now by ultraproduct argument.
\end{proof}

\begin{lemma}\label{lem:conn.int}
Let $\cG\to \cS$ be a group scheme and let $\mathcal{H} \Subset_\pi \mathcal{G}$ be a family of subgroups.
\begin{enumerate}
\item For any integer $d\in\N$, there exists a family of subgroups $\mathcal{K} \Subset_\tau \mathcal{G}$ such that, for any geometric point $s$ of $\mathcal{S}$ and any $d$ geometric points $s_1,\ldots,s_d$ of $\pi ^{-1} (s)$, the group $\mathcal{H}_{s_1} \cap \cdots \cap \mathcal{H}_{s_d}$ is of the form $\mathcal{K}_t$, where $t$ is a point over $s$ (i.e., $\tau(t)=s$).
\item There exists a family of subgroups $\mathcal{P} \Subset_\phi \mathcal{G}$ such that, for any geometric point $s$ of $\mathcal{S}$, any open subgroup of $\mathcal{G}_{s}$ is of the form $\mathcal{P}_t$, where $t$ is a point over $s$.
\end{enumerate}
\end{lemma}

\begin{proof}$ $	
\begin{enumerate}
\item Let $\cS'$  be the domain of definition of $\pi$ and let $\cG_{\cS'}:=\cG\times_\cS \cS'$. Define  $$\cK:=\calH \times_{\cG_{\cS'}} \cdots  \times_{\cG_{\cS'}}\calH,$$
and
$$\cS'':=\cS' \times_{\cS} \cdots  \times_{\cS'} \cS',$$
to be the $d$-fold fibered products. The natural maps $\mathcal{S}'' \rightarrow \mathcal{S}$ and $\cK\to \cS''$ give a subfamily as required.

\item After passing to a stratification of $\mathcal{S}$, we can assume that $\mathcal{G} \subset (\GL_n)_\mathcal{S}$ is closed. Fix a closed embedding $(\GL_n)_\mathcal{S} \hookrightarrow \mathbb{A}_{\mathcal{S}}^N$. 

Since $\mathcal{G} \rightarrow \mathcal{S}$ is of finite type, there is a bound $D$ on the complexity of all subgroups $\mathcal{G}_{s}$, where $s$ ranges over all geometric points of $\mathcal{S}$. By Lemma \ref{lem:ultra}, there is a bound $E$ on the complexity of all subgroups $(\mathcal{G}_{s})^\circ$, where $s$ ranges over all geometric points of $\mathcal{S}$. Since $\mathcal{G} \rightarrow \mathcal{S}$ is of finite type, there is a constant $C$ such that, for any geometric point $s$ of $\mathcal{S}$ the group  $\mathcal{G}_{s}$ has at most $C$ connected components. We get that there is a constant $M$ such that for any geometric point $s\in \cS$,  any open subgroup of $\mathcal{G}_s$ has complexity at most $M$.
	
There is a morphism $\mathcal{T} \rightarrow \mathcal{S}$ and a family of subgroups $\mathcal{P} \subset \mathcal{G}_\mathcal{T} \rightarrow \mathcal{T}$ such that the following holds: for every prime power $q$, every $s\in \mathcal{S}(\mathbb{F}_q)$, and every algebraic subgroup $\mathcal{P} \subset \mathcal{G}_s$ which is defined over $\mathbb{F}_q$ and has complexity at most $M$, we have that $\mathcal{P}=\mathcal{P}_t$, for some $t\in \mathcal{T}(\mathbb{F}_q)$. By the arguments above, this family satisfies the requirements.
\end{enumerate}
\end{proof}

\begin{proof}[Proof of Corollary \ref{cor:LP.non.ac}] 
It is enough to construct $\mathcal{K}$ and prove that the claim holds for all but finitely many primes. 
Let $\pi:\mathcal{S}'\rightarrow \mathcal{S}$ and $\mathcal{H} \rightarrow \mathcal{S}'$ be as in Theorem \ref{thm:LP} and let $d=\dim_\cS\, \mathcal{G}$.

By Lemma \ref{lem:conn.int}, we have a family $\mathcal{K}\Subset_{\tau} \mathcal{G}$, with $\tau:\cT\to \cS$, such that,  for any geometric point $s$ of $\mathcal{S}$ and any $d$ geometric points $s_1,\ldots,s_d$ of $\pi ^{-1} (s)$, any open subgroup of $\mathcal{H}_{s_1} \cap \cdots \cap \mathcal{H}_{s_d}$ is of the form $\mathcal{K}_t$, where $t$ is a point over $s$ (i.e., $\tau(t)=s$). We show that such a family $\mathcal{K}$ satisfies the conclusion of the Corollary.			

Let $q$ be a prime power, $s\in \mathcal{S}(\mathbb{F}_q)$, and $\Gamma \subset \mathcal{G}_s(\mathbb{F}_q)$ that $\mathbb{F}_q$-evades $\mathcal{K}$. Let $F: \pi ^{-1}(s)\to \pi ^{-1}(s)$ be the geometric Frobenius. We first show that $\Gamma$ $\overline{\mathbb{F}_q}$-evades $\mathcal{H}$. Assuming the contrary, there is $s' \in \pi ^{-1} (s)(\overline{\mathbb{F}_q})$ such that $\Gamma \subset \mathcal{H}_{s'}(\overline{\mathbb{F}_q})$. For every finite subset $I \subset \mathbb{Z}$, denote $\mathbf{H}_I=\bigcap_{i \in I} \mathcal{H}_{F^is'}$. There are $i_1,\ldots,i_d$ such that $\dim\, \mathbf{H}_{\{i_1,\ldots,i_d\}}=\min \left\{ \dim\, \mathbf{H}_I \mid I \subset \mathbb{Z} \text{ finite}\right\}$. The group $\mathbf{H}_{\{i_1,\ldots,i_d\}}^\circ$ is invariant under the Frobenius, so it is defined over $\mathbb{F}_q$. It follows that the group $\Gamma \mathbf{H}^{\circ}_{\{i_1,\ldots,i_d\}}$ is also defined over $\mathbb{F}_q$. By the assumption on $\mathcal{K}$, we have $\Gamma \mathbf{H}_{\{i_1,\ldots,i_d\}}^\circ = \mathcal{K}_t$, for some $t\in \mathcal{T}(\mathbb{F}_q)$. This contradicts the assumption that $\Gamma$ $\mathbb{F}_q$-evades $\mathcal{K}$.
			
The result now follows from Theorem \ref{thm:LP}.
\end{proof}

\begin{proposition} \label{prop:LP.sym} There is a function $\LP:\mathbb{N} \rightarrow \mathbb{N}$ such that, if $p$ is a prime number, $\mathbf{G}$ is a connected algebraic group over $\mathbb{F}_p$, $t:\mathbf{G} \rightarrow \mathbf{G}$ is an involution, and $\Gamma \subset \mathbf{G}(\mathbb{F}_p)$ is $t$-invariant, then there is a normal $t$-invariant subgroup $\Delta \triangleleft \Gamma$ of index at most $\LP(\dim(\mathbf{G}))$, a connected reductive group $\mathbf{H}$ defined over $\mathbb{F}_p$, an involution $s$ of $\mathbf{H}$, and an $S_2$-equivariant homomorphism $\rho : \Delta \rightarrow \mathbf{H}(\mathbb{F}_p)$ such that: \begin{enumerate} 
\item $\dim \mathbf{H} \leq \dim \mathbf{G}$.
\item $\ker \rho$ is a $p$-group.
\item If $\dim \, \mathbf{H} = \dim \, \mathbf{G}$, then $\ker \rho=1$.
\item $\rho(\Delta)$ is big in $\mathbf{H}(\mathbb{F}_p)$.
\end{enumerate} 
\end{proposition} 
For the proof we will need some preparations:

\begin{lemma} \label{lem:extending.central.extension} Let $p>3$, let $\bfH$ be a connected semi-simple group over $\F_p$, and let $1\to C\to E\to \mathbf{H}(\mathbb{F}_p)' \rightarrow 1$ be a finite central extension. Then there is a finite central extension $\mathbf{E}$ of $\mathbf{H}$ and an embedding $E \hookrightarrow \mathbf{E}(\mathbb{F}_p)$ such that the diagram
\begin{equation}\label{eq:com}
\xymatrix{E \ar[r] \ar@{^{(}->}[d] & \mathbf{H}(\mathbb{F}_p)' \ar@{^{(}->}[d] \\ \mathbf{E}(\mathbb{F}_p) \ar[r] & \mathbf{H}(\mathbb{F}_p)}
\end{equation}
commutes. In addition,
	\begin{enumerate}
\item\label{lem:extending.central.extension:1} If $\sigma$ is an automorphism of $\mathbf{H}$, $\tau$ is an automorphism of $E$ and the map $E \rightarrow \mathbf{H}(\mathbb{F}_p)'$ is equivariant, then there is an automorphism of $\mathbf{E}$ such that the map $\mathbf{E} \rightarrow \mathbf{H}$ is equivariant.
\item\label{lem:extending.central.extension:2} if $|C|$ is prime to $p$, then $E \cap \mathbf{E}^\circ$ is big in $\mathbf{E}^\circ(\mathbb{F}_p)$. 
\end{enumerate} 
\end{lemma}
\begin{proof} Let $\widetilde{\mathbf{H}}$ be the universal cover of $\mathbf{H}$. By Theorem \ref{thm:simp.con}, the universal central extension (or universal cover, cf. \cite[\S1]{Moo}) of $\mathbf{H}(\mathbb{F}_p)'$ is $\widetilde{\mathbf{H}}(\mathbb{F}_p)$. Denote the kernel of $\widetilde{\mathbf{H}}(\mathbb{F}_p) \rightarrow \mathbf{H}(\mathbb{F}_p)' $ by $A$ and note that $$A \subset Z(\widetilde{\mathbf{H}}(\mathbb{F}_p)) \overset{\text{\ref{thm:simp.con}\eqref{thm:simp.con:2.5}}}{=} Z(\widetilde{\mathbf{H}})(\mathbb{F}_p).$$ The extension $1 \rightarrow C \rightarrow E \rightarrow \mathbf{H}(\mathbb{F}_p)' \rightarrow 1$ corresponds to a homomorphism $\alpha:A \rightarrow C$. In particular, $E=\widetilde{\mathbf{H}}(\mathbb{F}_p) \times C / \delta(A)$, where $\delta(g)=(\alpha(g),g ^{-1})$. 

Consider $A$ and $C$ as discrete algebraic groups. Let $\mathbf{C}$ be the zero-dimensional algebraic group 
\[
\mathbf{C}=\left( C \times Z(\widetilde{\mathbf{H}}) \right)/\delta (A)
\]
and let 
\[
\mathbf{E}=\left( C \times \widetilde{\mathbf{H}} \right) / \delta(A).
\]
The short exact sequence $1 \rightarrow \mathbf{C} \rightarrow \mathbf{E} \rightarrow \mathbf{H} \rightarrow 1$ is a central extension, the map $E = \left( C \times \widetilde{\mathbf{H}}(\mathbb{F}_p)\right) / \delta(A) \rightarrow \mathbf{E}(\mathbb{F}_p)$ is injective, and the diagram \eqref{eq:com} commutes. 

It remains to prove the additional claims. Claim \eqref{lem:extending.central.extension:1} follows from the construction. For Claim \eqref{lem:extending.central.extension:2}, assume  that $|C|$ is prime to $p$. By the construction, $\mathbf{E}^\circ=\widetilde{\mathbf{H}}/\ker(\alpha)$. It follows that $E\cap \mathbf{E}^\circ \supset \widetilde{\mathbf{H}}(\mathbb{F}_p)/\ker(\alpha)$, so the map $E \cap \mathbf{E}^\circ \rightarrow \mathbf{H}(\mathbb{F}_p)'$ is onto. Since the order of $Z(\widetilde{\mathbf{H}})(\overline{\mathbb{F}_p})$ is prime to $p$, the same is true for the size of $\mathbf{C}$ and the size of the kernel of $E \cap \mathbf{E}^\circ \rightarrow \mathbf{H}(\mathbb{F}_p)'$. Since the kernel of the surjection $E \cap \mathbf{E}^\circ \rightarrow \mathbf{H}(\mathbb{F}_p)'$ is prime to $p$, we have that the number of $p$-elements of $E \cap \mathbf{E}^\circ$ is equal to the number of $p$-elements of $\mathbf{H}(\mathbb{F}_p)'$. By the same reasoning applied to the surjection $\mathbf{E}^\circ(\mathbb{F}_p) \rightarrow \mathbf{H}(\mathbb{F}_p)'$, this is also the number of $p$-elements of $\mathbf{E}^\circ(\mathbb{F}_p)$. Hence all $p$-elements in $\mathbf{E}^\circ(\mathbb{F}_p)$ are already in $E\cap \mathbf{E}^\circ$. By Theorem \ref{thm:simp.con} and Corollary \ref{cor:univ.cover}, $\mathbf{E}^\circ(\mathbb{F}_p)'$ is generated by its $p$-elements, and the second claim follows.
\end{proof}

The next lemma follows from \cite[Proposition 1.5.5, Theorem 1.5.6 (i)]{Mar} and Theorem \ref{thm:simp.con}\eqref{thm:simp.con:0}.
\begin{lemma}\label{lem:alm.pref}
Let $p>3$ and let $\bfH$ be a reductive group over $\F_p$. Then any $g\in \bfH(\F_p)$ that commutes with  $\bfH(\F_p)'$ is central in $\bfH$.
\end{lemma}


\begin{proof}[Proof of Proposition \ref{prop:LP.sym}] For every $n$, let $(\mathcal{R}_n,\mathcal{S}_n,\tau_n)$ be the versal family of reductive groups with involutions from Lemma \ref{lem:fam.of.sym}. By Lemma \ref{lem:fam.of.sym}\eqref{lem:fam.of.sym:3}, there is a subscheme $\mathcal{S}^s_n\subset \mathcal{S}_n$ such that, for any geometric point $x$ of $\mathcal{S}_n$, the group $(\cR_n)_x$ is absolutely simple adjoint iff $x$ factors through $\mathcal{S}^s_n$. Let $\mathcal{R}_n^s \subset \mathcal{R}_n$ be the preimage of $\mathcal{S}^s_n$.
	
Applying Corollary \ref{cor:LP.non.ac} to $\mathcal{R}^{s}_n \rightarrow \mathcal{S}^{s}_n$, we get a family of subgroups $\mathcal{K}_n \Subset \mathcal{R}^{s}_n$ parameterized by an $\mathcal{S}^{s}_n$-scheme $f:\mathcal{S}_n' \rightarrow \mathcal{S}^{s}_n$. Let $D(n)$ be the maximum of the number of connected components of a group of the form $(\mathcal{K}_n)_x \cap (\tau_n)_{f(x)}(\mathcal{K}_n)_x$, or of the form $(\mathcal{K}_n)_x$, where $x$ is a geometric point of $\mathcal{S}_n'$. 

We define three functions $\LP,\LPss,\LPAd:\mathbb{N} \rightarrow \mathbb{N}$ by recursion. Set 
\[
\LP(1)=\LPss(1)=\LPAd(1)=1
\]
and, for $n \geq 2$, set 
\[
\LPAd(n)=\max \left\{ (3^n+2)^n, \LP(n-1)D(n)^\co{n} \right\}.
\]

\[
\LPss(n)=2^n\LPAd(n).
\]

\[
\LP(n)=2^n \LPss(n).
\]


Note that $\LPAd \leq \LPss \leq \LP$.
We will show that the proposition holds with this choice of $\LP$. The proof is by induction on $n:=\dim\,\mathbf{G}$. The base of the induction, $n=0$, is trivial. The induction step is divided to the following steps: 
\begin{enumerate}[{\bf Step 1:}]
\item The claim holds if $p \leq 3^{n}+1$ with the bound $\LP$ replaced by $\LPAd$. \\ In this case, we can take $\Delta =1$, using the bound $|\mathbf{G}(\mathbb{F}_p)| \leq (p+1)^{\dim\, \mathbf{G}}$ from \cite[Lemma 3.5]{Nor}. 
\item  The claim holds if $\mathbf{G}$ is semisimple and adjoint, with the following improvements:
\begin{enumerate}
	\item The bound $\LP$ is replaced by  $\LPAd$. 
	\item Either $\bH$ is semisimple and $Ker(\rho)=1$, or there is a proper connected $t$-invariant subgroup $\bN<\bG$ such that $[\Gamma :\Gamma \cap  \bN(\F_p)]<D(n)$.\\
\end{enumerate} 

By Step 1, we can assume $p>3^{n}+1$. We have $\mathbf{G}=\mathbf{G}_1 \times \cdots \times \mathbf{G}_m$, where $\mathbf{G}_i$ are simple and adjoint. We denote the projection $\mathbf{G} \rightarrow \mathbf{G}_i$ by $pr_i$. Each $\mathbf{G}_i$ is a restriction of scalars from an absolutely simple (and adjoint) group: $\mathbf{G}_i=\Res_{\mathbb{F}_{q_i}/\mathbb{F}_p} \mathbf{S}_i$. 

For any $i\in \{1,\dots, m\}$, we will define a point  $\sigma_i\in \mathcal{S}^s_n(\mathbb{F}_{q_i})$. By Corollary \ref{cor:aut.perm}, for any $i$  we have $t(\mathbf{G}_i)=\mathbf{G}_j$, for some $j$. If $i\neq j$, we take  $\sigma_i\in \mathcal{S}^s_n(\mathbb{F}_{q_i})$ to be such that $(\mathcal{R}_n)_{\sigma_i}\cong \mathbf{S}_i$. This is posible since $\dim(\mathbf{S}_i)\leq \dim(\mathbf{G}_i)$. If $i= j$, we take  $\sigma_i=\sigma_j\in \mathcal{S}^s_n(\mathbb{F}_{q_i})$ to be such that $(\mathcal{R}_n)_{\sigma_i}\cong \mathbf{S}_i$ in an $S_2$-equivariant way.

In both cases, we identify $\mathbf{G}_i(\mathbb{F}_p)$, $\mathbf{S}_i(\mathbb{F}_{q_i})$, and $(\mathcal{R}_n)_{\sigma_i}(\mathbb{F}_{q_i})$. Given $\Gamma \subset \mathbf{G}(\mathbb{F}_p)$, there are two cases:
\begin{enumerate} [{Case} 1:]
\item For some $i$, $pr_i(\Gamma) \subset \mathbf{G}_i(\mathbb{F}_p)$ does not $\mathbb{F}_{q_i}$-evade $\mathcal{K}_n$.\\
For simplicity, we assume that $t(\mathbf{G}_i)=\mathbf{G}_i$; if $t(\mathbf{G}_i)=\mathbf{G}_j$, the proof is similar (and simpler). In this case, there is a point $x\in \cS'(\bF_{q_i})$ that lies over $\sigma_i$ such that $pr_i(\Gamma) \subset (\mathcal{K}_n)_x(\mathbb{F}_{q_i})$. Denote  $\mathbf K:= \mathcal{K}_x$. By the definition of $D(n)$, we have $|\pi_0(\mathbf K \cap \mathbf t_{\sigma_i}(\bK))|\leq D(n)$. Let $\mathbf{M}=\Res_{\mathbb{F}_{q_i}/\mathbb{F}_p}(\mathbf K \cap \mathbf t_{\sigma_i}(\bK))$. We have $|\pi_0(\mathbf M)|\leq D(n)^n$

Using the identification $\mathbf{G}_i=\Res_{\mathbb{F}_{q_i}/\mathbb{F}_p} \mathbf{S}_i$, the group $\mathbf{M}$ is a subgroup of $\mathbf{G}_i$ and is defined over $\mathbb{F}_p$. Note that $pr_i(\Gamma) \subset \mathbf{M}(\mathbb{F}_p)$. Since $\dim(pr_i ^{-1} (\mathbf{M}))<n$, the result now follows from the induction step applied to $pr_i ^{-1} (\mathbf{M})^\circ$  and $\Gamma \cap  pr_i ^{-1} (\mathbf{M})^\circ(\bF_p)$.

\item For all $i$, $pr_i(\Gamma)$ $\mathbb{F}_{q_i}$-evades $\mathcal{K}_n$. \\ 
In this case, there are Frobenius maps $\Phi_i: \mathbf{S}_i(\overline{\mathbb{F}_{q_i}}) \rightarrow \mathbf{S}_i(\overline{\mathbb{F}_{q_i}})$ such that $pr_i(\Gamma)$ is big in $\mathbf{S}_i(\overline{\mathbb{F}_{q_i}})^{\Phi_i}$. Let $\Delta:=\Gamma \cap \prod_i \left( \mathbf{S}_i(\overline{\mathbb{F}_{q_i}})^{\Phi_i} \right) '$. Since
\[
\left[ \mathbf{S}_i(\overline{\mathbb{F}_{q_i}})^{\Phi_i} : \left( \mathbf{S}_i(\overline{\mathbb{F}_{q_i}})^{\Phi_i} \right) '\right] \leq 2^{\dim\, \mathbf{S}_i},
\]
we get
\[
[ \Gamma : \Delta ] \leq 2^{\sum \dim\, \mathbf{S}_i} \leq 2^{\dim\, \mathbf{G}}\leq \LPAd(n).
\]

Since $\Delta \supset \Gamma'$, it follows that $pr_i(\Delta)=\left( \mathbf{S}_i(\overline{\mathbb{F}_{q_i}})^{\Phi_i} \right) '$, for all $i$. Since $\left( \mathbf{S}_i(\overline{\mathbb{F}_{q_i}})^{\Phi_i}\right)'$ are simple groups, Goursat's Lemma implies that there is a subset $I \subset [m]$ such that the projection $\Delta \rightarrow \prod_{i\in I} \left( \mathbf{S}_i(\overline{\mathbb{F}_{q_i}})^{\Phi_i} \right) '$ is an isomorphism. Since $\Delta$ is perfect and $\Delta \subset \Gamma$, it follows that $\Delta=\Gamma'$, and, in particular, $t(\Delta)=\Delta$. 

By \cite[11.6]{Ste68}, there is a connected semisimple $\mathbb{F}_p$ group $\mathbf{H}$ such that $\mathbf{H}(\mathbb{F}_p)\cong \prod_{i\in I} \mathbf{S}_i(\overline{\mathbb{F}_{q_i}})^{\Phi_i}$.

By Corollary \ref{cor:ext}, the restriction of $t$ to $\Delta$ extends to an involution $s$ of $\mathbf{H}$.

By \cite[Lemma 3.5]{Nor}, $(p-1)^{\dim\, \mathbf{H}} \leq \left| \mathbf{H}(\mathbb{F}_p) \right|$ and $\left| \mathbf{G}(\mathbb{F}_p) \right| \leq (p+1)^{\dim\, \mathbf{G}}$, so
\[
(p-1)^{\dim\, \mathbf{H}} \leq \left| \mathbf{H}(\mathbb{F}_p) \right| \leq | \Delta | 2^{\dim\,\mathbf{G}} \leq \left| \mathbf{G}(\mathbb{F}_p) \right|2^{\dim\, \mathbf{G}} \leq (2p+2)^{\dim\, \mathbf{G}} \leq (3p-3)^{\dim\, \mathbf{G}}.
\]
Since $p>3^{\dim\,\mathbf{G}}+1$, we have 
\[
\frac{\dim\, \mathbf{H}}{\dim\,\mathbf{G}} \leq \frac{\log(3p-3)}{\log(p-1)} < 1+\frac{1}{\dim\,\mathbf{G}},
\]
so $\dim\,\mathbf{H} \leq \dim\,\mathbf{G}$.
\end{enumerate} 

\item The claim holds if $\mathbf{G}$ is semisimple (but not necessarily adjoint), with the  bound  $\LPss$.\\
Let $\overline{\mathbf{G}}=\mathbf{G}/\mathbf{Z}(\mathbf{G})$, let $\pi:\mathbf{G} \rightarrow \overline{\mathbf{G}}$ be the projection, and let $\overline{\Gamma}=\pi(\Gamma)$. Applying the previous step to $\overline{\Gamma},\overline{\mathbf{G}}$, there are two possible cases:
\begin{enumerate}[{Case} 1:]
\item There is a proper connected $t$ -invariant subgroup $\overline \bN<\overline \bG$ such that $[\overline \Gamma :\overline \Gamma \cap  \overline \bN(\F_p)]<D(n)$.
\item There is a  subgroup $\overline{\Delta} \subset \overline{\Gamma}$, a semisimple group $\overline{\mathbf{H}}$ with an action of $S_2$, and an $S_2$-equivariant injective homomorphism $\overline{\rho}:\overline{\Delta} \rightarrow \overline{\bfH}(\mathbb{F}_p)$, such that $\overline{\rho}(\overline{\Delta})\subset \overline{\bfH}(\mathbb{F}_p)$ is big and $[\overline \Gamma:\overline\Delta]<\LPAd(n)$.
\end{enumerate}
In the first case, we are done by the induction assumption. For the second case, note that $\rho \left( \overline \Delta' \right)=\overline\bfH(\F_p)'$ and, by Lemma \ref{lem:der}, we have $[\overline \Delta:\overline \Delta']\leq 2^n.$
Denote  $\Delta:=\pi^{-1}(\overline \Delta') \cap \Gamma$. We get that $$[\Gamma:\Delta]<\LPAd(n)2^n=\LPss(n)$$
$\Delta$  is a central extension of  
$\overline{\Delta}'=\overline\bfH(\F_p)'$. By Lemma \ref{lem:extending.central.extension}, this central extension can be extended to an $S_2$-equivariant central extension $\bfH$ of $\overline \bfH$ in such a way that the embedding $$\Delta \to  \bfH(\F_p)$$ has big image, as required.

\item The claim holds if $\mathbf{G}$ is a direct  product of  a semisimple group and a torus, with the bound  $\LPss$.\\
Write $\mathbf{G}=\mathbf{G}'\times \bfT$. Let $p:\mathbf{G}\to\mathbf{G}'$ be the projection and $\bar \Gamma:=p(\Gamma)$.
Applying the previous step to $\bar \Gamma$ we get groups $\bar \Delta$ and  $\bar \bfH$. The claim holds for $\Delta:=\bar \Delta\times \bfT(\F_p)\cap \Gamma$ and $\bfH:=\bar \bfH\times \bfT$.

\item The claim holds if $\mathbf{G}$ is reductive.\\
Let $\tilde \bG:= \bG' \times Z(\bG)^\circ$.
We have an isogeny $\pi:\tilde \bG\to \bG$. By Lemma \ref{lem:coker} and Theorem  \ref{thm:simp.con}\eqref{thm:simp.con:3}, $[\bG(\F_p), \pi(\tilde \bG(\F_p))]\leq 2^n$. Let $\tilde \Gamma:=\pi^{-1}(\Gamma)$. Applying the previous step to $\tilde \Gamma\subset \tilde \bG(\F_p)$ we get groups  $\tilde \Delta$, $\tilde \bfH$, and a map $\tilde \rho:\tilde \Delta\to \tilde \bfH(\F_p)$. Note that $Ker(\pi)\cap \tilde \Delta$ is central in $\tilde \Delta$ and thus $\tilde\rho(Ker(\pi)\cap \tilde \Delta)$ is central in $\tilde\rho(\tilde \Delta)$. Since $\rho(\tilde \Delta)$ is big in $ \tilde \bfH(\F_p)$ we get (by Lemma \ref{lem:alm.pref}) that $\tilde\rho(Ker(\pi)\cap \tilde \Delta)$ is central in $\tilde \bfH$. Define $\bfH=\tilde \bfH/\tilde\rho(Ker(\pi)\cap \tilde \Delta)$ and $\Delta:=\pi(\tilde \Delta)\cong \tilde \Delta/(Ker(\pi)\cap \tilde \Delta)$. Note that $[\Gamma:\Delta]\leq \LP(n)$.
The map $\tilde \rho$ desends to a map $\rho:\Delta\to \bfH(\F_p)$  and we are done.
\item The claim holds for all connected groups $\mathbf{G}$.\\
Denoting the unipotent radical of $\mathbf{G}$ by $\mathbf{U}$, $t$ induces an involution on $\mathbf{G} / \mathbf{U}$ and the projection $\pi:\mathbf{G} \rightarrow \mathbf{G}/\mathbf{U}$ is equivariant. By the previous step, we can assume that $\mathbf{U}$ is positive dimensional. Given $\Gamma$, let $\bar \Gamma =\pi(\Gamma)$. By induction, there is a subgroup $\bar \Delta\subset \bar \Gamma$, an algebraic group $\bar \bfH$, and a homomorphism $\bar \rho: \bar \Delta \rightarrow \bfH (\mathbb{F}_p)$. It is easy to see that $\Delta := \pi ^{-1} (\bar \Delta)$, $\mathbf{H}:=\bar \bfH$, and $\rho:= \bar \rho \circ \pi$ satisfy the requirements of the proposition.
\end{enumerate}


\end{proof} 




\begin{corollary} \label{cor:LP.ind} There is an increasing function $C_{mon}: \mathbb{N} \rightarrow \mathbb{N}$ for which the following holds. If $p$ is a prime and $\Delta \subset \Gamma$ are finite groups such that $\Rad_p(\Delta)\neq \Delta\cap \Rad_p(\Gamma)$, then there is a normal subgroup $\Delta^\circ \lhd \Delta$ of index at most $C_{mon0}(\ralg_p( \Gamma ))$ satisfying $\ralg_p(\Delta^\circ) <  \ralg_p (\Gamma)$.
\end{corollary}
For the proof  we will need the following:
\begin{lemma}\label{lem:no.p} Let $p > 3$ be a prime and let $\mathbf{H}$ be a reductive group over $\mathbb{F}_p$. Let $\Gamma<\mathbf{H}(\mathbb{F}_p)$ be a big subgroup. Then  $\Gamma$ does not have a non-trivial normal $p$-subgroup. 
\end{lemma}
\begin{proof} Suppose that $P$ is a non-trivial normal $p$-subgroup of $\Gamma$. Since the index of $\mathbf{H}(\mathbb{F}_p)'$ in $\mathbf{H}(\mathbb{F}_p)$ is prime to $p$, we have  $P\subset \mathbf{H}(\mathbb{F}_p)'$. For similar reasons, $P\cap Z(\mathbf{H}(\mathbb{F}_p)')=1$. This gives an embedding of $P$ into  $\mathbf{H}(\mathbb{F}_p)'/Z(\mathbf{H}(\mathbb{F}_p)')$, which is a product of non-abelian simple groups, a contradiction.
\end{proof}

\begin{proof}[Proof of Corollary \ref{cor:LP.ind}] Set $\mon(n):=\max(4^n,(\LP(n)+1)^n)$. Suppose that $\Delta \subset \Gamma$ are as in the statement of the corollary. Without loss of generality, we can assume that $\Rad_p(\Gamma)$ is trivial. Let $n:=\ralg_p(\Gamma)$.  Using the bound $|\mathbf{G}(\mathbb{F}_p)| \leq (p+1)^{\dim\, \mathbf{G}}$ from \cite[Lemma 3.5]{Nor}, we may also assume that $p>3$ and $p>\LP(n)$ (otherwise, the claim holds with $\Delta^\circ=1$). Embed $\Gamma \hookrightarrow \mathbf{G}(\mathbb{F}_p)$ with $\mathbf{G}$ connected reductive of dimension $n$. Applying Proposition \ref{prop:LP.sym} to $\Delta \subset \mathbf{G}(\mathbb{F}_p)$, there is a normal subgroup $\Delta^\circ \triangleleft \Delta$, a connected reductive group $\mathbf{H}$ defined over $\mathbb{F}_p$, and a homomorphism $\rho: \Delta ^\circ \rightarrow \mathbf{H}(\mathbb{F}_p)$ such that
\begin{enumerate}
\item $[\Delta : \Delta ^\circ] \leq \LP(n)$.
\item $\ker \rho$ is a $p$-group.
\item $\rho(\Delta ^\circ)$ is big in $\mathbf{H}(\mathbb{F}_p)$.
\item $\dim\,\mathbf{H} \leq n$.
\item If $\dim\,\mathbf{H} = n$ then $\ker\rho=1$.
\end{enumerate} 
By Lemma \ref{lem:no.p}, $\ker \rho = \Rad_p(\Delta ^\circ)$. In particular, $\ralg_p(\Delta^\circ) \leq \dim\,\mathbf{H}$.

If $\dim\,\mathbf{H} < n$, we are done. Otherwise, since $[\Delta : \Delta ^\circ] < \LP(n) < p$, we get that 
\[
\Rad_p(\Delta) \subset \Rad_p(\Delta ^\circ)=\ker \rho=1,
\]
contradicing the condition that $\Rad_p(\Delta)\neq \Delta \cap \Rad_p(\Gamma)$.

\end{proof} 

\begin{corollary} \label{cor:LP.reductive.with.involution} There is an increasing function $\LPsym: \mathbb{N} \rightarrow \mathbb{N}$ for which the following holds. If $p$ is a prime, $(\Gamma,\theta)$ is a finite symmetric pair, and $\Rad_p(\Gamma)=1$, then there is a $\theta$-invariant normal subgroup $\Delta \triangleleft \Gamma$ satisfying $[\Gamma : \Delta]<\LPsym(\alg_p(\Gamma))$, a connected reductive algebraic group $\mathbf{H}$ satisfying $\dim\, \mathbf{H} \leq 2\alg_p(\Gamma)$, an involution $t$ of $\mathbf{H}$, and an $S_2$-equivariant embedding $\Delta \subset \mathbf{H}(\mathbb{F}_p)$ such that $\mathbf{H}(\mathbb{F}_p)' \subset \Delta \subset \mathbf{H}(\mathbb{F}_p)$.
\end{corollary} 

\begin{proof} Set $\LPsym(n)=\LP(2n).$ Embed $\Gamma \hookrightarrow \mathbf{G}$ and apply Proposition \ref{prop:LP.sym} to $\mathbf{G} \times \mathbf{G}$, the involution $t(x,y)=(y,x)$, and the subgroup $\left\{ (x,\theta(x)) \mid x\in \Gamma \right\}\cong \Gamma$.
\end{proof}

\subsection{$\theta$-invariant subgroups of bounded index}\label{ssec:S2}
In this subsection we prove Corollary \ref{cor:LP.ind.sym}, which is a $S_2$-equivariant version of the monotonicity of $\ralg$, Corollary \ref{cor:LP.ind}.

\begin{lemma}\label{lem:norm.sub}
There is a function $C_{inv}:\N \times \N  \to \N$ which is increasing in both variables such that, for any
\begin{itemize}
\item pair of finite groups  $\Delta<\Gamma$ and
\item a prime $p$,
\end{itemize}
there exists a subgroup $\Delta^\circ \lhd \Delta$ which is normal in $ \Gamma$ and satisfies
 $$[\Gamma:\Delta^\circ] \leq C_{inv}(\ralg_p(\Delta),[\Gamma:\Delta])$$ and $$\ralg_p(\Delta^\circ)\leq  \ralg_p(\Delta).$$
\end{lemma}
\begin{proof}
Define recursively $$C_{inv}(0,k)=k!$$ and 
$$C_{inv}(n,k):=k!\mon(n)C_{inv}(n-1, k!\mon(n)).$$
 We will prove the lemma by induction on $\ralg_p(\Delta)$. For $\ralg_p(\Delta)=0$, the claim is clear. For the induction step, let $n>0$ be an integer, and assume the lemma holds if
$\ralg_p(\Delta)<n$. We prove the lemma for $\ralg_p(\Delta)=n$. Let

$$\Delta_1:=\bigcap_{\gamma\in \Gamma}\gamma \Delta \gamma^{-1}.$$ We consider the following cases:
\begin{enumerate} [{Case} 1.]
\item $\Rad_p(\Delta_1)= \Rad_p(\Delta)\cap \Delta_1$.\\
In this case $\ralg_p(\Delta_1)\leq\ralg_p(\Delta)=n$ so we can take 
$\Delta^\circ:= \Delta_1$ and we are done.
\item $\Rad_p(\Delta_1)\neq \Rad_p(\Delta)\cap \Delta_1$.\\
In this case, Corollary \ref{cor:LP.ind} implies that we can find $\Delta_2$ such that 
$$\ralg_p(\Delta_2)<\ralg_p(\Delta)=n$$ 
and
$$[\Delta_1:\Delta_2]\leq \mon(n).$$ 
By the induction hypothesis, there exists $\Delta^\circ \lhd \Delta_2$  which is normal in $\Gamma$ and satisfies
 $$[\Delta_2:\Delta^\circ] \leq C_{inv}(n-1,[\Gamma:\Delta_2])$$ and $$\ralg_p(\Delta^\circ)\leq \ralg_p(\Delta_2)<n.$$
We get 
\begin{align*}
[\Gamma:\Delta^\circ] &\leq  [\Gamma:\Delta_1]\cdot [\Delta_1:\Delta_2]\cdot [\Delta_2:\Delta^\circ]
\\&\leq 
[\Gamma:\Delta]!  \mon(n) C_{inv}(n-1,[\Gamma:\Delta_2])
\\&\leq 
[\Gamma:\Delta]!  \mon(n) C_{inv}(n-1,[\Gamma:\Delta_1]\cdot [\Delta_1:\Delta_2])
\\&\leq 
[\Gamma:\Delta]!  \mon(n) C_{inv}(n-1,[\Gamma:\Delta]!  C_{LP}(n))=C_{inv}(n,[\Gamma:\Delta]).
\end{align*}
\end{enumerate} 
\end{proof}

The last lemma and Corollary \ref{cor:LP.ind} imply:
\begin{corollary}\label{cor:LP.ind.sym} There is a function $\monSym:\N \to \N$ such that, for any odd prime $p$, any finite group $\Gamma$, any subgroup $\Delta < \Gamma$ satisfying $\Rad_p(\Delta)\neq \Delta\cap \Rad_p(\Gamma)$, and any involution $\theta$ of $\Delta$,
there is  a normal $\theta$-invariant subgroup $\Delta^\circ\lhd \Delta$ such that
\begin{itemize}
\item $\ralg_p(\Delta^\circ) <  \ralg_p (\Gamma).$
\item $[\Delta:\Delta^\circ]\leq \monSym(\ralg_p(\Gamma))$
\end{itemize}
%
\end{corollary}
\begin{proof}
Set:
$$\monSym(n):=C_{inv}(n,2\mon(n)).$$
By the monotonicity of the $\ralg_p$ (Corollary \ref{cor:LP.ind}), we can find $\Delta_1\lhd \Delta$ satisfying 
$$\ralg_p(\Delta_1) <  \ralg_p (\Gamma)$$ and
$$[\Delta:\Delta_1]\leq \mon(\ralg_p(\Gamma)).$$
Let $\tilde \Delta:=\langle \theta \rangle \ltimes \Delta$. 
By Lemma  \ref{lem:norm.sub}, there is a normal subgroup $\Delta^\circ \lhd \Delta_1$ which is also normal in $\tilde \Delta$ and satisfies
 $$[\tilde \Delta:\Delta^\circ] \leq C_{inv}(\ralg_p(\Delta_1),[\tilde \Delta:\Delta_1])$$ and $$\ralg_p(\Delta^\circ)\leq  \ralg_p(\Delta_1)<\ralg_p(\Gamma).$$
 The fact that $\Delta^\circ$ is normal in  $\tilde \Delta$ implies that $\Delta^\circ$ is $\theta$-invariant.
 We also have:
 \begin{align*}
[\Delta:\Delta^\circ]&\leq [\tilde \Delta:\Delta^\circ] \leq C_{inv}(\ralg_p(\Delta_1),[\tilde \Delta:\Delta_1])
\\&\leq
C_{inv}(\ralg_p(\Gamma),[\tilde \Delta:\Delta_1])
=
C_{inv}(\ralg_p(\Gamma),2[\Delta:\Delta_1])
\\&\leq
C_{inv}(\ralg_p(\Gamma),2\mon(\ralg_p(\Gamma)))= \monSym(\ralg_p(\Gamma))
 \end{align*}
 \end{proof}

\section{Groups of odd order} \label{sec:odd} 
In this section we analyze symmetric pairs of groups of odd order and prove several statements about them that will be used in the proof of the main theorem. In particular, we prove a strong version of the main theorem for symmetric pairs of groups of odd order. Namely, we prove that they are all Gelfand pairs (Corollary \ref{cor:gel.Lap.Pra}\eqref{cor:gel.Lap.Pra:1}).

We also prove some other results for symmetric pairs of groups of odd order, see Lemma \ref{lem:H1odd} and Corollary \ref{cor:gel.Lap.Pra}\co{\eqref{cor:gel.Lap.Pra:2}}.

\begin{remark} $ $ \begin{enumerate}
\item Even though the results in this section are valid for arbitrary groups of odd order, we will only use them for $p$-groups for odd $p$.
\item All the proofs in this sections are based on the fact that, for a group of odd order $n$, the map $x\mapsto x^{\frac{n+1}{2}}$ is a square root.
\end{enumerate} 
\end{remark}

\begin{lemma} \label{lem:H1odd} If $\Omega$ is a finite group of odd order and $\theta$ is its involution, then $H^1(S_2,\Omega)=1$.
\end{lemma}
\begin{proof}
We need to show that, for every element $s\in \Omega$ that satisfies $s=\theta(s^{-1})$, there exists an element $g\in \Omega$ such that $s=\theta(g^{-1})g$.

Let $g=s^{\frac{|\Omega|+1}{2}}.$ Then $$\theta(g^{-1})g=\theta((s^{\frac{|\Omega|+1}{2}})^{-1}) s^{\frac{|\Omega|+1}{2}}=\theta(s^{-1})^{\frac{|\Omega|+1}{2}} s^{\frac{|\Omega|+1}{2}}=s^{\frac{|\Omega|+1}{2}} s^{\frac{|\Omega|+1}{2}}=s^{|\Omega|+1}=s.$$
\end{proof}

\begin{lemma}[Gelfand-Kazhdan Property for symmetric pairs of odd order]\label{lem:odd.GK} If $\Omega$ is a finite group of odd order and $\theta: \Omega \rightarrow \Omega$ is an involution, then, for any $g\in \Omega$, there are $h_1,h_2\in \Omega^\theta$ such that $h_1 g h_2=\theta(g^{-1})$.
\end{lemma}
Lemma \ref{lem:odd.GK} follows immediately from the following:
\begin{lemma}[polar decomposition for symmetric pairs of odd order]\label{lem:odd.small} If $\Omega$ is a finite group of odd order and $\theta: \Omega \rightarrow \Omega$ is an involution, then
	$$\Omega=\Omega^{\theta} \cdot \Omega^{\theta\circ inv},$$ where $\Omega^{\theta\circ inv}:=\{g\in \Omega|\theta(g)=g^{-1}\}$.
\end{lemma}
\begin{proof}
Let $g\in \Omega$. Define $s=(\theta(g^{-1})g)^{\frac{|\Omega|+1}{2}}$ and $o=gs^{-1}$.
Then,
\begin{align*}
\theta(s)&=\theta((\theta(g^{-1})g)^{\frac{|\Omega|+1}{2}})\\
            &= (\theta(\theta(g^{-1})g))^{\frac{|\Omega|+1}{2}}\\
            &=(g^{-1})\theta(g))^{\frac{|\Omega|+1}{2}}\\
            &=((\theta(g^{-1}) g)^{-1})^{\frac{|\Omega|+1}{2}}\\
            &=((\theta(g^{-1}) g)^{\frac{|\Omega|+1}{2}})^{-1})=s^{-1},
\end{align*}
\begin{align*}
\theta(o)&=\theta(gs^{-1})=\theta(g)s\\
            &=\theta(g)(\theta(g^{-1})g)^{\frac{|\Omega|+1}{2}}\\
            &=\theta(g)\theta(g^{-1})g(\theta(g^{-1})g)^{\frac{|\Omega|+1}{2}-1}\\
            &=g(\theta(g^{-1})g)^{\frac{|\Omega|+1}{2}-1}\\
            &=g(\theta(g^{-1})g)^{\frac{|\Omega|-1}{2}}\\
            &=g(\theta(g^{-1})g)^{|\Omega|-\frac{|\Omega|+1}{2}}\\
            &=g(\theta(g^{-1})g)^{-\frac{|\Omega|+1}{2}}=gs^{-1}=o\\
\end{align*}
and, thus, 
$$g=gs^{-1}s=os.$$
\end{proof}

Lemma \ref{lem:odd.GK} gives the following:
\begin{corollary}\label{cor:gel.Lap.Pra} If $\Gamma$ is a finite group of odd order and $\theta$ is an involution of $\Omega$, then
\begin{enumerate}
\item\label{cor:gel.Lap.Pra:1} (Gelfand property for symmetric pairs of odd order:) $(\Omega,\Omega^ \theta)$ is a Gelfand pair, i.e., for any $\rho\in \irr(\Omega)$, $$\dim\rho^{\Omega^\theta}\leq 1.$$
\item\label{cor:gel.Lap.Pra:2} (Lapid-Prasad property for symmetric pairs of odd order:) Any representation $\rho$ of $G$ which is $\Omega^ \theta$ distinguished (i.e. satisfies $\dim\rho^{\Omega^\theta}>0$) also satisfies $$\rho^*\simeq\rho\circ \theta.$$
\end{enumerate}
\end{corollary}

\begin{proof} The claims follow from Lemma \ref{lem:odd.GK} and \cite[Theorem 2]{Aiz_GK} (which is an adaptation of results from \cite{GK} and \cite{JR}).
\end{proof}

\section{bounds on $H^1(S_2,\Gamma)$}\label{sec:H1}
In this section we prove Corollary \ref{cor:h1.her} which gives bounds on the cohomology group $H^1(S_2,\Delta)$, for a finite symmetric pair $(\Delta,\theta)$.

\begin{lemma}\label{lem:h1r}
There is an increasing  function $C_{H1R0}:\N\to \N$  such that, for any finite field $F$ of odd characteristic, any connected reductive group $\bG$, and any involution $t$ of $\bG$, both defined over $F$,
   $$|H^1(S_2,\bG(F))|<C_{H1R0}(\dim \bG).$$
\end{lemma}
\begin{proof} Let $n$ be an integer. Let $\Phi_n:\cR_n\to \cS_n$ and $\tau_n$  be as in Lemma \ref{lem:fam.of.sym}. Let $s_n$ be the  anti automorphism $s_n(g):=\tau_n(g^{-1})$. Let $C_{H1R0}(n)$ be such that, for any geometric point $x$ of $\cS_n$, 
\[
\left|\pi_0 \left( \Phi_n^{-1}(x)^{s_n} \right)\right |<C_{H1R0}(n).
\]
Fix $x\in \cS_n(F)$ and let $\bG:=\Phi_n^{-1}(x)$. We need to show that $$|H^1(S_2,\bG(F))|\leq C_{H1R0(n)}.$$

By definition we have $H^1(S_2,\bG(F))=\bG(F)^{s_n}/\bG(F)$ where the action of  $\bG$ on $\bG^{s_n}$ is given by $g \cdot x:=gx s_n(g).$ Since $\bG$ is connected, Lang's theorem implies  that the map $\bG(F)^{s_n}/\bG(F)\to (\bG^{s_n}/\bG)(F)$ is an embedding. By  analyzing the action of the Lie algebra it is easy to see that the orbits of the action of $\bG$ on  $\bG^{s_n}$ are open. Thus, $$|H^1(S_2,\bG(F))|=|\bG(F)^{s_n}/\bG(F)|\leq  |(\bG^{s_n}/\bG)(F)|\leq |(\bG^{s_n}/\bG)(\overline{F})|=|\pi_0(\bG_{\overline{F}}^{s_n})|\leq C_{H1R0}(n).$$
\end{proof}

{
The following lemmas are straightforward:
\begin{lemma}\label{lem:h1ab}
Let $A$ be a finite abelian group with an action of $S_2$. Denote $A[2]=\left\{ x\in A \mid x^2=0 \right\}$. Then $|H^1(S_2,A)|\leq |A[2]|^2$.
\end{lemma}

\begin{proof} Let $\theta\in \Aut(A)$ be the non-trivial element of $S_2$. Denote $C=\left\{ x\in A \mid x \theta(x)=1 \right\}$, $B=\left\{ x ^{-1} \theta(x) \mid x\in A \right\}$, and $S=\left\{ x^2 \mid x\in A \right\}$. We need to show that $|C/B| \leq |A[2]|^2$. Since $|C/B| = |C/C\cap S| \cdot |C \cap S / B\cap S| \leq |A[2]|\cdot |C \cap S / B\cap S|$, it in enough to show that $|C \cap S / B\cap S| \leq |A[2]|$.

Let $x\in C\cap S$, and let $z\in A$ be such that $z^2=x$. Then $z \theta(z)\in A[2]$ and $x \left( z ^{-1} \theta(z) \right) = z \theta(z)\in A[2]$. Hence, $C\cap S \subseteq A[2] \cdot B$, so $|C \cap S / B\cap S| \leq |A[2] \cdot B / B| \leq |A[2]|$.
\end{proof} 

\begin{lemma} \label{lem:h1.ses} Let $\Gamma$ be a finite group, let $\theta$ be an involution of $\Gamma$, and let $N\triangleleft \Gamma$ be a normal $\theta$-invariant subgroup. Then \begin{enumerate}
\item $|H^1(S_2, N)| \leq |H^1(S_2,\Gamma)| \cdot [\Gamma : N]$. 
\item 
\[
|H^1(S_2,\Gamma)| \leq |H^1(S_2,\Gamma/N)| \cdot \max_{\tau}|H_\tau^1(S_2,N)|,
\]
where $\tau$ ranges over all involutions of $N$ (including $\tau=1$) and $H_\tau ^1(S_2,N)$ is the cohomology of the $S_2$-module $N$ given by the involution $\tau$.
\end{enumerate} 
\end{lemma} 

\begin{proof} In the proof, if $G$ is a group with involution $\sigma$, we identify $H^1(S_2,G)$ with the quotient of $C_{G,\sigma}:=\left\{ x\in G \mid x \sigma(x)=1 \right\}$ by the action of $G$ given by $g \cdot x= \sigma(g)x g ^{-1}$. 
\begin{enumerate}
\item We have have a map $\alpha:C_{N,\theta} / N \rightarrow C_{\Gamma,\theta} / \Gamma$ taking $N \cdot x$ to $\Gamma \cdot x$. Given $x\in C_{N,\theta}$, the fiber $\alpha ^{-1} (\alpha (N \cdot x))=\alpha ^{-1} (\Gamma \cdot x)$ is equal to $(\Gamma \cdot x) / N$, so its size is at most $[\Gamma : N]$. Hence, $|C_{N,\theta} / N| \leq |C_{\Gamma,\theta} / \Gamma | \cdot [\Gamma : N]$.
\item We have a map $\beta:C_{\Gamma,\theta}/\Gamma \rightarrow C_{\Gamma /N,\theta}/ (\Gamma/N) $ sending $\Gamma \cdot x$ to $(\Gamma / N) \cdot (xN)$. It is enough to show that the sizes of the fibers of $\beta$ are bounded by $\max_{\tau}|H_\tau^1(S_2,N)|$.

For $x\in C_{\Gamma,\theta}$, let $\tau_x:N \rightarrow N$ be the automorphism $\tau_x(n)=x ^{-1} \theta(n) x$. Since $\tau_x^2(n)=x ^{-1} \theta(x ^{-1}) n \theta(x)x=n$, the automorphism $\tau_x$ is an involution. 

Suppose that $y\in C_{\Gamma,\theta}$ and $\beta(\Gamma \cdot x)=\beta(\Gamma \cdot y)$. Then there is $y'\in \Gamma \cdot y$ such that $y'=x n$, for some $n\in N$. Since 
$$1=\theta(y')y'=\theta(x)\theta(n)x n=\theta(x) x x ^{-1} \theta(n) x n= \tau_x(n)n,$$ we get that $n\in C_{N,\tau_x}$. Hence $\beta ^{-1} (\beta(\Gamma \cdot x) )$ can be identified with $xC_{N,\tau_x}/ \Gamma$. For any $m\in N$, 
\[
m \cdot y'=\theta(m) x n m ^{-1} = x \tau_x(m) n m ^{-1},
\]
so $|\beta ^{-1} (\beta(\Gamma \cdot x) )| \leq |xC_{N,\tau_x} / N |=|H^1_{\tau_x}(S_2,N)|$.
\end{enumerate} 
\end{proof} 
}
Now Lemma \ref{lem:h1r}	implies the following:

\begin{cor}\label{cor:h1r} There is an increasing  function $C_{H1R}:\N\to \N$  such that, for any prime $p$, any connected semi-simple group $\bG$, defined over $\F_p$, any involution $\theta$ of $\bG(\F_p)'$, we have 
$$|H^1(S_2,\bG(\F_p)')|<C_{H1R}(\dim \bG).$$		
	\end{cor}
	
\begin{proof}
Take $C_{H1R}(n)=C_{H1R0}(\dim \bG)4^{\dim G}$, where $C_{H1R0}$ is the function given by Lemma \ref{lem:h1r}.

Using the bound $|\mathbf{G}(\mathbb{F}_p)| \leq (p+1)^{\dim\, \mathbf{G}}$ from \cite[Lemma 3.5]{Nor}, we can  assume that $p>3$.
		
Let $\bar \bG:=\bG/Z(\bG)$. By Corollary \ref{cor:univ.cover}, the induced map $\bG(\F_p)'\to \bar \bG(\F_p)'$ is onto, and its kernel is $Z(\bG(\F_p)')$.  Let $\bar\theta$ be the involution induced by $\theta$ on $\bar \bG(\F_p)'$. 		
By Theorem \ref{thm:class}, we can lift $\bar \theta$ to an involution $t$ of $\bar \bG$. By Lemmas \ref{lem:h1.ses}, \ref{lem:h1r},
\ref{lem:der}  and Theorem \ref{thm:simp.con}, we have:
\begin{align*}
		|H^1(S_2,\bG(\F_p)')|&
		\overset{\ref{lem:h1.ses}}{\leq}
		 |H^1(S_2,\bar \bG(\F_p)')| \max_\tau |H^1_\tau(S_2,Z(\bG(\F_p)'))|
		{\leq}
		|H^1(S_2,\bar \bG(\F_p)')| \cdot |Z(\bG(\F_p)')|\\&
		\overset{\ref{thm:simp.con}}{\leq}   
		|H^1(S_2,\bar \bG(\F_p)')| 2^{\dim G}
		\overset{\ref{lem:h1.ses}}{\leq}
		|H^1_t(S_2,\bar \bG(\F_p))|\cdot [\bar \bG(\F_p):\bar \bG(\F_p)'] 2^{\dim G}\\&
		\overset{\ref{lem:der}}{\leq}
		|H^1_t(S_2,\bar \bG(\F_p))|\cdot 4^{\dim G}
		\overset{\ref{lem:h1r}}{<}   
			C_{H1R0}(\dim \bG)4^{\dim G}=C_{H1R}(\dim \bG).		
\end{align*}
		
\end{proof}

Now we can derive our bound on the first cohomology:
\begin{proposition}\label{prop:h1}
There is a function $C_{H1}:\N\to \N$  such that for any finite group $\Gamma$,  any involution $\theta$ of $\Gamma$, and any prime $p>2$  
we have
   $$H^1(S_2,\Gamma)<C_{H1}(\ralg_p(\Gamma)).$$
\end{proposition}
\begin{proof}
We take $$C_{H1}(n):=C_{H1R}(2n)\LPsym(n)2^{8n},$$
where $C_{H1R}$ is the function given by Corollary \ref{cor:h1r} and $\LPsym$ is the function given by Corollary \ref{cor:LP.reductive.with.involution}.

Using the bound $|\mathbf{G}(\mathbb{F}_p)| \leq (p+1)^{\dim\, \mathbf{G}}$ from \cite[Lemma 3.5]{Nor} and Lemmas \ref{lem:H1odd}, \ref{lem:h1.ses}, we may assume that $p>3.$

Let $\bar\theta$ is the involution of $\Gamma/\Rad_p(\Gamma)$ induced by $\theta$.  
By Corollary \ref{cor:LP.reductive.with.involution} there is a $\bar\theta$-invariant normal subgroup $\Delta$ of $\Gamma/\Rad_p(\Gamma)$, a connected reductive group $\bH$ defined over $\F_p$, and an involution $t$ of $\bH$ such that
\begin{itemize}
\item $[\Gamma/\Rad_p(\Gamma):\Delta]\leq C_{LP}(\ralg(\Gamma))$,
\item $\dim\, \bH\leq 2\ralg(\Gamma)$,
\item There is an equivariant embedding $(\Delta,\bar\theta|_\Delta)\hookrightarrow (\bH(\F_p),t)$,
\item $\bH(\mathbb{F}_p)'<\Delta< \bH(\mathbb{F}_p)$.
\end{itemize}
{By Corollary \ref{cor:perfect.prime}, we have $\Delta'=\mathbf{H}(\mathbb{F}_p)'$.}

Applying  Theorem \ref{thm:simp.con}, Corollary \ref{cor:h1r} and Lemmas \ref{lem:H1odd},
\ref{lem:h1ab}, \ref{lem:h1.ses}, we get

\begin{align*}
|H^1(S_2,\Gamma)| &\overset{\ref{lem:h1.ses}}{\leq} |H^1(S_2,\Gamma/\Rad_p(\Gamma))|
		 \cdot \max_{\tau} |H^1_\tau(S_2,\Rad_p(\Gamma))|\overset{\ref{lem:H1odd}}{=}
		 |H^1(S_2,\Gamma/\Rad_p(\Gamma))| 
            \\&\overset{\ref{lem:h1.ses}}{\leq} \max_{\tau}|H^1_\tau(S_2,\Delta)| \cdot |H^1(S_2,(\Gamma/(\Rad_p(\Gamma))/\Delta))|  
      		     \\&\overset{\ref{lem:h1.ses}}{\leq}
      		     \max_{\tau} |H^1_{\tau}(S_2,\Delta)| \cdot [\Gamma/\Rad_p(\Gamma):\Delta]
		     \\&\overset{\quad\,\,\,\,\,}{\leq}
		     \max_{\tau}|H^1_\tau(S_2,\Delta)| C_{LP}(\ralg(\Gamma))
		     \\& \overset{\ref{lem:h1.ses}}{\leq}
		      \max_{\sigma} |H^1_\sigma(S_2,\Delta')| \cdot \max_{\tau} |H^1_\tau(S_2,\Delta / \Delta')| \cdot C_{LP}(\ralg(\Gamma))
		     \\& 		     \overset{\ref{lem:h1ab}}{\leq}
		      \max_{\sigma} |H^1_\sigma(S_2,\mathbf{H}(\mathbb{F}_p)')| \cdot | \left( \Delta / \Delta' \right)[2]|^2 \cdot C_{LP}(\ralg(\Gamma))
		     \\& \overset{\ref{cor:h1r}}{\leq}
		     C_{H1R}(2\ralg(\Gamma)) \cdot | \left( \Delta / \Delta' \right)[2]|^2 \cdot C_{LP}(\ralg(\Gamma))
		     \\&\overset{\quad\,\,\,\,\,}{\leq}
		     C_{H1R}(2\ralg(\Gamma)) \cdot | \left( \mathbf{H}(\mathbb{F}_p) / \mathbf{H}(\mathbb{F}_p)' \right)[2]|^2 \cdot C_{LP}(\ralg(\Gamma))
		     \\&\overset{\quad\,\,\,\,\,}{\leq}
		     C_{H1R}(2\ralg(\Gamma)) \cdot | 
		     \left( \mathbf{H}(\mathbb{F}_p) / \mathbf{H}'(\mathbb{F}_p) \right)[2]|^2   
		     \cdot [ \mathbf{H}'(\mathbb{F}_p) : \mathbf{H}(\mathbb{F}_p)' ]^2\cdot C_{LP}(\ralg(\Gamma))
		     \\&\overset{\ref{lem:der}}{\leq}
		     C_{H1R}(2\ralg(\Gamma)) \cdot | 
		     \left( (\mathbf{H}/\mathbf{H}')(\mathbb{F}_p)\right)[2]|^2 
		     \cdot  2^{4\ralg(\Gamma)}\cdot C_{LP}(\ralg(\Gamma))
		     \\&\overset{\quad\,\,\,\,\,}{\leq}
		      C_{H1R}(2\ralg(\Gamma)) \cdot  2^{8\ralg(\Gamma)}\cdot C_{LP}(\ralg(\Gamma))=C_{H1}(\ralg(\Gamma))		     
\end{align*}

\end{proof}

\begin{corollary}\label{cor:h1.her} There is an increasing function $ C^{her}_{H1}:\N\to \N$ such that, for any pair of finite groups $\Delta \subset \Gamma$,  any involution $\theta$ of $\Delta$, and any prime $p>2$, 
	$$|H^1(S_2,\Delta)|< C^{her}_{H1}(\ralg_p(\Gamma)).$$
\end{corollary}
\begin{proof}
Define $$C^{her}_{H1}(n):=\monSym(n) C_{H1}(n).$$ By Corollary \ref{cor:LP.ind.sym}, there is a $\theta$-invariant normal subgroup $\Delta^\circ\lhd \Delta$ such that $\ralg_p(\Delta^\circ)\leq \ralg_p(\Gamma)$ and $[\Delta:\Delta^\circ]\leq \monSym(\ralg_p(\Gamma))$. The previous proposition (Proposition \ref{prop:h1}) implies that $|H^1(S_2,\Delta^\circ)|< C_{H1}(\ralg_p(\Gamma)).$ From the exact sequence:
	$$H^1(S_2,\Delta^\circ)\to H^1(S_2,\Delta)\to H^1(S_2,\Delta/\Delta^\circ)$$
we get
	\begin{align*}
		|H^1(S_2,\Delta)| &\leq |H^1(S_2,\Delta^\circ)|\cdot |H^1(S_2,\Delta/\Delta^\circ)|
		\\&\leq
		 C_{H1}(\ralg_p(\Gamma)) [\Delta:\Delta^\circ]
		\\&\leq
		 C_{H1}(\ralg_p(\Gamma)) \monSym(\ralg_p(\Gamma))
		 \\&=
		C^{her}_{H1}(\ralg_p(\Gamma)).
	\end{align*}	
\end{proof}

\section{bounds on $H^2(\Gamma,\mu_{p^\infty})$}  \label{sec:H2}

Let  $\Bmu_{p^\infty}$ denote the group of roots of unity of order which is a power of $p$.
In this section we prove that $H^2(\Gamma,\Bmu_{p^\infty})$ is trivial whenever $\Rad_p(\Gamma)=1$ and $p$ is large with respect to 
$\alg_p(\Gamma)$. See Proposition \ref{prop:H2} below.

We will need the following:
\begin{lemma}\label{lem:H2}
For any short exact sequence of finite groups
$$1\to \Gamma_1\to \Gamma_2\to \Gamma_3\to 1,$$
we have 
\begin{enumerate}
\item  \label{lem:H2:1} If $p \nmid | \Gamma_1 |$, then $H^i(\Gamma_2,\mathbb{F}_p)\cong H^i(\Gamma_3,\mathbb{F}_p)$, for all $i$.
\item\label{lem:H2:2} If $p \nmid | \Gamma_3 |$, then $H^i(\Gamma_2,\mathbb{F}_p)\cong H^i(\Gamma_1,\mathbb{F}_p)^{\Gamma_3}$, for all $i$.
\end{enumerate}
\end{lemma}
\begin{proof}
For a pair finite group $\Gamma\lhd \Gamma'$, let $I^{\Gamma'}_{\Gamma}:\cM_{\F_p}(\Gamma') \to \cM(\Gamma'/\Gamma)$  be the functor of $\Gamma$-invariants from the category of $\F_p$-representations of $\Gamma'$ to the category of $\F_p$-representations of $\Gamma'/\Gamma$. Note that, if $p \nmid |\Gamma|$, then $I^{\Gamma'}_{\Gamma}$ is exact. Also denote $$I_{\Gamma}:=I^{\Gamma}_{\Gamma}.$$
Now
\begin{enumerate}
\item If $p \nmid | \Gamma_1 |$, then 
\begin{align*}
H^i(\Gamma_2,\mathbb{F}_p) 
&
\cong R^i(I_{\Gamma_2})(\F_p) \cong R^i(I^{}_{\Gamma_3} \circ I^{\Gamma_2}_{\Gamma_1})(F_p) \cong R^i(I^{}_{\Gamma_3}) \circ I^{\Gamma_2}_{\Gamma_1} (\F_p) \cong
\\&\cong
 R^i(I^{}_{\Gamma_3}) (\F_p)   \cong H^i(\Gamma_3,\mathbb{F}_p).
\end{align*}
 
\item If $p\nmid | \Gamma_3 |$, then 
\begin{align*}
H^i(\Gamma_2,\mathbb{F}_p)&\cong R^i(I_{\Gamma_2})(\F_p)\cong R^i(I^{}_{\Gamma_3} \circ I^{\Gamma_2}_{\Gamma_1})(F_p)
\\&\cong
I^{}_{\Gamma_3} \circ R^i I^{\Gamma_2}_{\Gamma_1} (\F_p) \cong I^{}_{\Gamma_3} (H^i(\Gamma_1,\mathbb{F}_p))   \cong 
H^i(\Gamma_1,\mathbb{F}_p)^{\Gamma_3}.
\end{align*}
\end{enumerate}
\end{proof}

Now we can prove the main result of this section:

\begin{prop}[vanishing of $H^2$ for large $p$]\label{prop:H2}
There is an increasing function $C_{H2}:\N\to \N$  such that for any 
\begin{enumerate}
\item integer $n$,
\item  prime $p>C_{H2}(n)$,
\item  finite group $\Gamma$ such that $\Rad_p(\Gamma)=1$ and $\alg_p(\Gamma)\leq n$,
\end{enumerate}
the group $H^2(\Gamma,\Bmu_{p^\infty})$ is trivial.  
\end{prop}

\begin{proof} 
It is enough to show the claim after replacing $\Bmu_{p^ \infty}$ by $\mathbb{F}_p$.
Define $$C_{H2}(n):=\max(3,\LP(n),4^{n}).$$
Fix $n$.
Let   $p>C_{H2}(n)$   be a prime and 
 $\Gamma$  be a finite group such that $\Rad_p(\Gamma)=1$, and $\alg_p(\Gamma)\leq n$.
 we have to show that 
 $$H^2(\Gamma,\mathbb{F}_p)=0.$$

By Corollary \ref{cor:LP.reductive.with.involution}
applied with trivial involution\footnote{in fact Proposition  \ref{prop:LP.sym} is enough}, there are a normal subgroup $\Delta \lhd \Gamma$, a connected reductive algebraic group $\mathbf{H}$ defined over $\mathbb{F}_p$, and an injective homomorphism $\rho: \Delta \rightarrow \mathbf{H}(\mathbb{F}_p)$ such that 
\begin{enumerate}
\item $[\Gamma:\Delta] \leq \LPsym(n)$.
\item $\bH(\F_p)'< \rho(\Delta) < \bH(\F_p)$.
\item $\dim \bH\leq 2n$.
\end{enumerate}
Applying \ref{lem:H2}\eqref{lem:H2:2} to the embedding $\Delta \subset \Gamma$, it is enough to prove that $H^2(\Delta,\mathbb{F}_p)=0$. We will identify $\rho(\Delta)$ with $\Delta$.

Let $\Delta_0=\Delta\cap \bfH'(\mathbb{F}_p)$. We have an embedding of $\Delta/\Delta_0$ into $(\bfH/\bfH')(\mathbb{F}_p)$. Thus $p \nmid [\Delta,\Delta_0]$. By descent of $H^2$ to a subgroup (Lemma \ref{lem:H2}\eqref{lem:H2:2}) this implies that it is
 engh to show that 
 $$H^2(\Delta_0,\mathbb{F}_p)=0.$$

By Lemma \ref{lem:der}, $[\Delta_0:\mathbf{H}'(\mathbb{F}_p)'] \leq 4^n$. By Lemma \ref{lem:H2}\eqref{lem:H2:2}, it is enough to show that $H^2(\mathbf{H}'(\mathbb{F}_p)',\mathbb{F}_p)=0$.

Let $\pi: \widetilde{\mathbf{H}'} \to \mathbf{H}'$ the universal cover. By Corollary \ref{cor:univ.cover}, 

\[
H^2(\mathbf{H}'(\mathbb{F}_p)',\mathbb{F}_p)=H^2(\pi(\widetilde{\mathbf{H}'}(\mathbb{F}_p)),\mathbb{F}_p)
\]
and the latter group vanishes by combining  descent of $H^2$ to a quotient (Lemma \ref{lem:H2}\eqref{lem:H2:1})   and  vanishing of $H^2$ for simply connected groups (Theorem \ref{thm:simp.con}\eqref{thm:simp.con:1}).

\end{proof}

\section{The case of trivial $p$-radical}\label{sec:triv.prad}
In this section we prove a twisted version of the main result of the paper for the special case when $\Gamma$ have a trivial $p$-radical (See Corrolary \ref{cor:nr} below).
	
We start with the case that $\Gamma$ is a finite group of Lie type. This case, in the larger generality of spherical pairs but without a twist, was proved in \cite{She}. A variation of the argument of \cite{She} proves the twisted case. We include this variation in Appendix \ref{sec:Shai}. In particular,
the following is a special case of Theorem \ref{thm:Shai}:
\begin{theorem}\label{thm:Shai.sp} There is an increasing function $C_{rd}:\N\to \N$	such that,	for every finite field $F$ of characteristic $>3$, every connected reductive group $\bG$, and every involution $t$ of $\bG$, we have
	$$\mu(\bG(F),t)<C_{rd}(\dim(\bG)).$$
\end{theorem}
\Rami{In order to apply this theorem to arbitrary groups with trivial $p$-radical, we will need the following:}

\begin{lemma}[{cf. \cite[Lemma 3.2.1]{AA_bnd}}]\label{lem:mult.iso}
Let $\phi:\Gamma_1\to \Gamma_2$ be morphism finite groups and let $\theta_1,\theta_2$  be  involutions of $\Gamma_1,\Gamma_2$ such that $\theta_2\circ \phi=\phi \circ \theta_1$. Then,
\begin{enumerate}         
\item\label{lem:mult.iso:1} $\mu(\Gamma_2,\theta_2) \leq [\Gamma_2 : \phi(\Gamma_1)] \mu(\Gamma_1,\theta_1).$
\item\label{lem:mult.iso:2} $\nu(\Gamma_2,\theta_2) \leq [\Gamma_2 : \phi(\Gamma_1)] \nu(\Gamma_1,\theta_1).$
\end{enumerate}
\end{lemma}


\begin{lemma} \label{lem:bound.center.derived} For every prime $p$ and every connected reductive group $\mathbf{H}$ over $\mathbb{F}_p$, we have $\left[ \mathbf{H}(\mathbb{F}_p) : \mathbf{H}(\mathbb{F}_p)' \cdot Z(\mathbf{H})(\mathbb{F}_p) \right] \leq 2^{2\dim\, \mathbf{H}}$.
\end{lemma} 

\begin{proof} The map $\phi : \mathbf{H}' \times Z(\mathbf{H}) \rightarrow \mathbf{H}$ is an isogeny. By Lemma \ref{lem:coker} and Theorem \ref{thm:simp.con}, 
\[
\left[ \mathbf{H}(\mathbb{F}_p) : \mathbf{H}'(\mathbb{F}_p) \cdot Z(\mathbf{H})(\mathbb{F}_p) \right] \leq |(\ker \phi)(\overline{\mathbb{F}_p})|=|Z(\mathbf{H}'(\overline{\mathbb{F}_p})| \leq |Z(\widetilde{\mathbf{H}'}(\overline{\mathbb{F}_p}))|\leq 2^{\dim\, \mathbf{H}'},
\]
where $\widetilde{\mathbf{H}'}$ is the universal cover of $\mathbf{H}'$. By Lemma \ref{lem:der}, $[\mathbf{H}'(\mathbb{F}_p):\mathbf{H}(\mathbb{F}_p)'] \leq 2^{\dim\, \mathbf{H}'}$, and the result follows.
\end{proof} 

The following lemma is straightforward
\begin{lemma} \label{lem:mu.product} If $A,B$ are finite groups and $\theta_A,\theta_B$ are involutions of $A,B$, then $\mu(A \times B, \theta_A \times \theta_B)=\mu(A,\theta_A)\mu(B,\theta_B)$.
\end{lemma} 


\begin{corollary}\label{cor:nr} There is an increasing function $C_{nr}: \mathbb{N} \rightarrow \mathbb{N}$ such that, for any prime $p>3$ and any finite group $\Gamma$ that has a trivial $p$-radical, we have $\mu(\Gamma)<C_{nr}(\alg(\Gamma))$.
\end{corollary} 

\begin{proof} 
\Rami{Set $$C_{nr}(n):=\LPsym(n)16^{n}C_{rd}(2n).$$} Let $\theta$ be an involution of $\Gamma$. Applying Corollary \ref{cor:LP.reductive.with.involution} to $\Gamma$, we get a triple $\Delta,\mathbf{H},t$. By Lemma \ref{lem:mult.iso}, we have
\[
\mu(\Gamma,\theta) \leq [\Gamma:\Delta] \cdot \mu(\Delta,\theta) \leq \LPsym(\alg_p(\Gamma))\mu(\Delta,\theta).
\]
Denote $S=\mathbf{H}(\mathbb{F}_p)'$ and $Z=Z(\mathbf{H})(\mathbb{F}_p)$. Note that $S$ and $Z$ are $t$-invariant subgroups of $\mathbf{H}(\mathbb{F}_p)$. By Lemma \ref{lem:bound.center.derived}, 
\[
\mu(\Delta,\theta)\leq [\Delta:\Delta \cap S \cdot Z] \mu(\Delta \cap S \cdot Z,\theta) \leq 2^{\dim\, \mathbf{H}} \mu (\Delta \cap S \cdot Z,\theta) \leq 4^{\alg_p(\Gamma)} \mu (\Delta \cap S \cdot Z,\theta).
\]
Let $\varphi:S \times Z \rightarrow H(\mathbb{F}_p)$ be the multiplication map. $\varphi$ is equivariant if we use the involution $\theta \times \theta$ on $S \times Z$. By Lemma \ref{lem:mult.iso}\eqref{lem:mult.iso:1}, 
\[
\mu(\Delta \cap S \cdot Z,\theta) \leq \mu(\varphi ^{-1} (\Delta), \theta \times \theta).
\]
Since $S \subset \Delta$, we get that $\varphi ^{-1} (\Delta)=S \times A$, for some $A \subset Z$. Therefore, by Lemma \ref{lem:mu.product},
\[
\mu(\varphi ^{-1} (\Delta),\theta \times \theta) \leq \mu(S,\theta).
\]
Finally, by \Rami{Lemma \ref{lem:der}},
\[
\mu(S,\theta) \leq \left[ \bH'(\mathbb{F}_p):S\right]\mu(\bH'(\mathbb{F}_p),t) \leq 2^{\dim(\bH')} \mu(\bH'(\mathbb{F}_p),t)\leq 4^{\alg_p(\Gamma)} \mu(\bH'(\mathbb{F}_p),t),
\]
and the result follows from Theorem \ref{thm:Shai.sp}.
\end{proof} 

\section{Clifford theory}\label{sec:Clif}
We recall the elements of Clifford theory. The following lemma is standard.

\begin{lemma}\label{lem:clif.ind}
Let $\Gamma$ be a finite group, let $\rho$ be an irreducible representation of $\Gamma$, and let $N \lhd \Gamma$ be a normal subgroup of $\Gamma$. Let $\tau$ be an irreducible sub-representation of $\rho\restriction_N$ and let $\sigma$ be the $\tau$-isotypic component of $\rho\restriction_N$. Let 
$\Delta:=\Gamma_\tau$ be the stabilizer of $\tau\in \irr(N)$ with respect to the adjoint action. Then $\sigma$ is $\Delta$-invariant and $$\rho\cong \ind_{\Delta}^{\Gamma}\sigma.$$
\end{lemma}

\begin{lemma}\label{lem:clif.H2}
Let $\Gamma$ be a finite group, let $\rho$ be an irreducible representation of $\Gamma$, and let $N \lhd \Gamma$ be a normal $p$-subgroup of $\Gamma$, and $\rho\in \irr(\Gamma)$. Assume that $\rho\restriction_N$ is isotypic and that $H^2(\Gamma/N,\Bmu_{p^\infty})$ is trivial. 

Then there exist $\pi_1,\pi_2\in \irr(\Gamma)$ such that $\pi_1\restriction_N$ is irreducible, the action of $N$ on $\pi_2$ is trivial, and $\rho\simeq \pi_1\otimes \pi_2$
\end{lemma}
\begin{proof} Write $\rho \restriction_N= \tau^{\oplus C}$, where $(\tau,V) \in Irr(N)$. Recall the construction of the 2-cocycle corresponding to the triple $(\Gamma,N,\tau)$: choose a set of coset representatives $T \subset \Gamma$ such that $1\in T$. For every $t\in T\smallsetminus \left\{ 1 \right\}$, choose an isomorphism $A_t: \tau^t \rightarrow \tau$  such that $\det(A_t)=1$, and let $A_1=I$. 
	
Define a  map $\pi:\Gamma\to \End_{\C}(V)$ by $\pi(tn)=A_t \tau(n)$ for $t\in T$ and $n\in N$. This is a projective representation of $\Gamma$  that extands $\tau$ and satisfies $\pi(tn)=\pi(t) \pi(n)$ and 
$$\pi(nt)=\pi(t t^{-1}nt)=A_t \tau(t^{-1}nt)=A_t\tau^{t}(n)=\tau(n) A_t=\pi(n) \pi(t).$$
 For every $\gamma_1,\gamma_2\in \Gamma$, the map $$\pi(\gamma _2 ^{-1})\pi(\gamma_1 ^{-1}) \pi(\gamma_1 \gamma_2 )$$ is an intertwiner of $\tau$, so $$\pi(\gamma _2 ^{-1})\pi(\gamma _1 ^{-1}) \pi( \gamma _1 \gamma _2 ) = \alpha(\gamma _1,\gamma _2) I,$$ for some $\alpha(\gamma_1,\gamma_2)\in \mathbb{C} ^ \times$. The map $\alpha:\Gamma\times \Gamma\to \C^\times$ is a 2-cocycle of $\Gamma$ with coefficients in $\mathbb{C} ^ \times$ and a simple computation shows that $\alpha$ descends to a 2-cocycle on $\Gamma / N$. Since $N$ is a $p$-group, $\det(\tau(n))\in \Bmu_{p^\infty}$, for every $n\in N$. Since $\det(A_t)=1$, for every $t\in T$, we get that $\alpha(\gamma_1,\gamma_2)^{\dim\, \tau}\in \Bmu_{p ^ \infty}$. Since $\dim\, \tau$ is a p-th power, we get that $\alpha(\gamma_1,\gamma_2)\in \Bmu_{p^ \infty}$. 

By assumption, this implies that $\alpha$ is cohomologous to the trivial cocycle, and, therefore, there is a representation $\tau_1$ of $\Gamma$ such that $\tau_1\restriction_{N}=\tau$. Since $\tau$ is irreducible, so is $\tau_1$. By \cite[Corollary 6.17]{Isa}, $\rho=\tau_1 \otimes \tau_2$, for some $\tau_2\in \Irr(\Gamma / N)$.
\end{proof}

\section{Proof of the main theorem}\label{sec:pf.main}
In this section we prove Theorem \ref{thm:main} and deduce Theorem \ref{thm:main.str} and Corollary \ref{cor:main}.

\begin{lemma}[the main theorem for product case] \label{lem:split.case} Let $C_{nr}$ be the increasing  function given by Corollary \ref{cor:nr}. Then, for every 
\begin{itemize}
\item prime $p>2$,
\item symmetric pair of finite groups  $(\Gamma,\Gamma^\theta)$,
\item $\rho_1 \in \irr(\Gamma)$ such that ${\rho_1}|_{\Rad_p(\Gamma)}$ is irreducible,
\item $\rho_2 \in \irr(\Gamma)$ that factors through $\Gamma/\Rad_p(\Gamma)$,
\end{itemize}
we have
$$\dim \left( (\rho_1\otimes \rho_2)^{\Gamma^\theta} \right)<C_{nr}(\ralg(\Gamma)).$$
\end{lemma}
\begin{proof}
Since odd order symmetric pairs are Gelfand pairs (Corollary \ref{cor:gel.Lap.Pra}\eqref{cor:gel.Lap.Pra:1}), we have  $$\dim \rho_1^{Rad_p(\Gamma)^\theta}\leq 1.$$ 
If $(\rho_1\otimes \rho_2)^{\Gamma^\theta}=0$, the claim trivially holds. Otherwise, let $\chi$ be the character with which $\Gamma^\theta/Rad_p(\Gamma)^\theta$ acts on  $\rho_1^{Rad_p(\Gamma)^\theta}$. Using Corollary \ref{cor:nr},
$$\dim(\rho_1\otimes \rho_2)^{\Gamma^\theta}=\dim(\chi\otimes \rho_2)^{\Gamma^\theta/Rad_p(\Gamma)^\theta}\leq \mu(\Gamma/Rad_p(\Gamma))\leq C_{nr}(\alg(\Gamma/Rad_p(\Gamma)))\leq C_{nr}(\ralg(\Gamma)).$$
\end{proof}

Using Clifford theory, the last lemma implies

\begin{corollary}\label{cor:nup} There is an increasing function $C_{\nu'}:\N\to \N$ such that, for every prime $p>2$ and any finite group $\Gamma$, we have
	$$\nu'_p(\Gamma)<C_{\nu'}(\ralg_p(\Gamma)),$$
where $\nu'_p$ is the function defined in Notations \ref{nota:mu.nu}.
\end{corollary}
\begin{proof}
	Define $C_{\nu'}(n)=\max((C_{H2}(n)+1)^{n},C_{nr}(n))$.
	We will prove the corollary by analyzing two cases:
	\begin{enumerate}[{Case} 1.]
	\item $p\leq C_{H2}(\ralg(\Gamma))$.\\
	In this case
	$$|\Gamma/Rad_p(\Gamma)|\leq (p+1)^{\ralg(\Gamma)}\leq (C_{H2}(\ralg(\Gamma))+1)^{\ralg(\Gamma)}\leq C_{\nu'}(\ralg(\Gamma)).$$ 
	Since we can control multiplicties when we pass to subgroup of bounded index (Lemma \ref{lem:mult.iso}), we get that $$\nu(\Gamma)\leq |\Gamma/Rad_p(\Gamma)| \cdot \nu(Rad_p(\Gamma)).$$
	Since odd order symmetric pairs are Gelfand pairs (Corollary \ref{cor:gel.Lap.Pra}\eqref{cor:gel.Lap.Pra:1}), we have that
	$$\nu(Rad_p(\Gamma))=1.$$
	We obtain
	$$\nu'_p(\Gamma)\leq \nu(\Gamma)\leq \#(\Gamma/Rad_p(\Gamma)) \leq C_{\nu'}(\ralg(\Gamma))$$
	\item $p> C_{H2}(\ralg(\Gamma))$.\\
By Proposition \ref{prop:H2}, the group 
	$H^2(\Gamma,\Bmu_{p^\infty})$ is trivial. Let $\rho\in \irr(\Gamma)$ be such that $\rho|_{Rad_p(\Gamma)}$ is isotypic. By Clifford theory (Lemma \ref{lem:clif.H2}), there exist $\pi_1,\pi_2\in \irr(\Gamma)$ such that $\pi_1|_{Rad_p(\Gamma)}$ is irreducible, the action of $Rad_p(\Gamma)$ on $\pi_2$ is trivial, and $\rho\simeq \pi_1\otimes \pi_2$. Thus, the main theorem for product case (Lemma \ref{lem:split.case}) implies that, for any involution $\theta$ of $\Gamma$, we have $$\rho^{\Gamma^\theta}= (\pi_1\otimes \pi_2)^{\Gamma^\theta}\leq C_{nr}(\ralg(\Gamma))\leq C_{\nu'}(\ralg(\Gamma)).$$
	\end{enumerate}
\end{proof}



We are now ready to prove the main theorem:

\begin{proof}[Proof of Theorem \ref{thm:main}] Define $C$ recursively by
\begin{enumerate}
\item $C(0):=1$
\item $C(n):=C^{her}_{H1}(n)\max(C_{\nu'}(n),C(n-1)\monSym(n))$,
\end{enumerate}
where $C^{her}_{H1}$ is given by Corollary \ref{cor:h1.her}, $C_{\nu'}$ is given by Corollary \ref{cor:nup}, and $\monSym$ is given by Corollary \ref{cor:LP.ind.sym}.
	We will prove the theorem by induction. We assume the theorem holds if $\ralg_p(\Gamma)<n$ and prove it in the case $\ralg(\Gamma)=n$. 
	
	Let $\theta$ be an involution of $\Gamma$ and $\rho\in \Irr \Gamma$. Let $N:=Rad_p(\Gamma)$. If $(\rho|_N)^{N^\theta}=0$, we are done. Otherwise, let $\tau$  be an irreducible direct summand of $\rho|_N$ such that $(\tau)^{N^\theta}\neq 0$. By the Lapid--Prasad property for symmetric pairs of odd order (Corollary \ref{cor:gel.Lap.Pra}\eqref{cor:gel.Lap.Pra:2}), we have $(\tau)^{\theta}\cong  \tau^*$.
	Let 
	$\Delta:=\Gamma_\tau$. We have $$\theta(\Delta)=\Gamma_{\tau^\theta}=\Gamma_{\tau^*}=\Gamma_{\tau}=\Delta,$$ showing that $\Delta$  is $\theta$-stable.
	
	By Clifford theory (Lemma \ref{lem:clif.ind}), we have $$\rho=ind_\Delta^\Gamma(\sigma),$$ where $\sigma \in \irr(\Delta)$ is such that $\sigma|_N$ is $\tau$-isotypic. Therefore,
\begin{equation}\label{eq:rho.inv}
	\rho^{\Gamma^\theta}=(ind_\Delta^\Gamma(\sigma))^H= \oplus_{[g]\in \Delta \backslash \Gamma / \Gamma^\theta} \sigma^{\Delta \cap (\Gamma^\theta)^g}= \oplus_{[g]\in \Delta\backslash \Gamma^{dis}/\Gamma^\theta } \sigma^{\Delta \cap (\Gamma^\theta)^g},
\end{equation}
	where 
	\begin{multline*}
	\Gamma^{dis} = \{g\in \Gamma| \sigma^{\Delta\cap (\Gamma^\theta)^g}  \neq 0\} 
	\subset  \{g\in \Gamma| \sigma^{N\cap  (\Gamma^\theta)^g}  \neq 0\}
	=   \{g\in \Gamma| \tau^{N\cap  (\Gamma^\theta)^g}  \neq 0\}=\\=
	\{g\in \Gamma| (\tau^g)^{N\cap  (\Gamma^\theta)}  \neq 0\}=\{g\in \Gamma| (\tau^g)^{N^\theta}  \neq 0\}\subset \{g\in \Gamma| (\tau^g)^* \circ \theta \simeq \tau^g\} =:\Gamma^{inv}
	\end{multline*}
	The last inclution is by the Lapid-Prasad property for symmetric pairs of odd order (Corollary \ref{cor:gel.Lap.Pra}\eqref{cor:gel.Lap.Pra:2}).

We analyze the quotient $\Delta \backslash \Gamma^{inv}/\Gamma^\theta$. Since $(\tau^g)^* \circ \theta=(\tau^*\circ \theta)^{\theta(g)}=\tau^{\theta(g)}$, we have $$\Gamma^{inv}=\{g\in \Gamma| \tau^{\theta(g)} \simeq \tau^g\}=\{g| \tau^{g\theta(g^{-1})} \simeq \tau \}=\{g|g\theta(g^{-1})\in\Delta \}.$$

	Thus, for any $g\in\Gamma^{inv}$,  we have $$\theta(ad(g^{-1})(\Delta))=ad(\theta(g^{-1}))(\Delta)=
	ad(g^{-1}g\theta(g^{-1}))(\Delta)=ad(g^{-1})(\Delta),$$ 
	and hence $ad(g^{-1})(\Delta)$ is $\theta$-stable.
	
	We also have $$\Gamma^{inv}/\Gamma^\theta= \{g\theta{g^{-1}\in \Delta|g\in \Gamma}\},$$ and thus,
	$$\Delta\backslash \Gamma^{inv}/\Gamma^\theta
	=Ker(H^1(S_2,\Delta)\to H^1(S_2,\Gamma)).$$
	Combining the last equality with \eqref{eq:rho.inv}, we get
	\begin{align*}
	\dim \rho^{\Gamma^\theta}&\leq
		|H^1(S_2,\Delta)| \max_{g\in \Gamma^{inv}} \dim \sigma^{\Delta \cap (\Gamma^{\theta})^g}
		\\&
		=   |H^1(S_2,\Delta)|   \max\limits_{g\in \Gamma^{inv}} \dim (\sigma\circ ad(g))^{(ad(g^{-1})(\Delta)) \cap \Gamma^\theta}
		\\&
		=   |H^1(S_2,\Delta)|  \max\limits_{g\in \Gamma^{inv}} \dim (\sigma\circ ad(g))^{(ad(g^{-1})(\Delta))^\theta}
		\\&
		=  |H^1(S_2,\Delta)|  \max\limits_{g\in \Gamma^{inv}} \dim \sigma^{\Delta^{ad(g)\circ\theta\circ ad(g^{-1})}}
		\\&
		\leq |H^1(S_2,\Delta)|\max\limits_{\theta' \text{ is an involution of } \Delta} (\dim (\sigma^{\Delta^{\theta'}})).
	\end{align*}

By our bound on $H^1(S_2,\Delta)$ (Corollary  \ref{cor:h1.her}),
	\begin{equation}\label{eq:bnd.on.mul}
	\dim \rho^{\Gamma^\theta}\leq C_{H1}^{her}(\ralg_p(\Gamma))\max\limits_{\theta' \text{ is an involution of } \Delta}  (\dim (\sigma^{\Delta^{\theta'}})).
	\end{equation}

	We finish the proof by analyzing the following cases:
	\begin{enumerate}[{Case} 1.]		
		\item $\Rad_p(\Gamma)\neq \Rad_p(\Delta)$:\\
		In order to bound $\nu(\Delta)$, we fix an involution $\theta'$  of $\Delta$.
		By Corollary \ref{cor:LP.ind.sym}, we can find a $\theta'$-invariant subgroup $\Delta^\circ\lhd \Delta$ such that
		 $$\ralg_p(\Delta^\circ)<\ralg_p(\Gamma)$$ and 
		 $$[\Delta:\Delta^\circ] \leq \monSym(\ralg_p(\Gamma)).$$
		 The induction assumption implies that $\nu(\Delta^\circ)\leq C(\ralg_p(\Gamma)-1)$. 
		 Lemma \ref{lem:mult.iso}\eqref{lem:mult.iso:2} implies that 
 		\begin{align*}
		\nu(\Delta,\theta')&\leq  [\Delta:\Delta^\circ]\cdot \nu(\Delta^\circ,\theta'|_{\Delta^\circ})
		\leq
			\monSym(\ralg_p(\Gamma)) C(\ralg_p(\Gamma)-1),
		\end{align*}
		so $$\nu(\Delta) \leq \monSym(\ralg_p(\Gamma)) C(\ralg_p(\Gamma)-1).$$ 
		
It follows that
		$$
		\dim \rho^{\Gamma^\theta}\leq  C_{H1}^{her}(\ralg_p(\Gamma ))\cdot   \monSym(\ralg_p(\Gamma)) C(\ralg_p(\Gamma)-1)\leq C(\ralg_p(\Gamma )),$$
		as required.
		 
		\item  $\Rad_p(\Gamma)= \Rad_p(\Delta)$.\\
		%
		%
		In this case $\ralg_p(\Delta)\leq \ralg_p(\Gamma)$, and $\sigma|_{\Rad_p(\Delta)}$ is isotypic.
		By \eqref{eq:bnd.on.mul} and Corollary \ref{cor:nup},

		\begin{align*}
			\dim \rho^{\Gamma^\theta}&\leq  C^{her}_{H1}(\ralg_p(\Gamma))\cdot  \nu'(\Delta)
			\\&\leq
			C^{her}_{H1}(\ralg_p(\Gamma)) C_{\nu'}(\ralg_p(\Delta))
			\\&\leq
			C^{her}_{H1}(\ralg_p(\Gamma)) C_{\nu'}(\ralg_p(\Gamma))\leq
			C(\ralg_p(\Gamma)).		
		\end{align*}
	\end{enumerate}
	
\end{proof}
\subsection{Deduction of Theorem \ref{thm:main.str} and Corollary \ref{cor:main}}\label{ssec:ded.str}
We will now deduce Theorem \ref{thm:main.str} from our main result, first reminding its formulation.
\NextVer{Name entered manualy}
\begin{customthm}{A}
There is an increasing function $C^{fin}:\N\to \N$
such that, for any
\begin{itemize}
\item odd prime $p$,
\item positive integer $d$
\item finite group $\Gamma$,
\item normal $p$-subgroup $N\lhd \Gamma$,
\item an embedding $\Gamma/N \hookrightarrow GL_d(\F_p)$,
\item an involution $\theta$ of $\Gamma$, 
\item an  irreducible representation $\rho$ of $\Gamma$,
\end{itemize}
we have $\dim \rho^{ \Gamma^\theta} \leq C^{fin}(d)$.
\end{customthm}
\begin{proof}
Set
$$C^{fin}(d):=\monSym(d^2)C(d^2).$$
By Corollary \ref{cor:LP.ind.sym}, we have a subgroup $\Delta\lhd \Gamma/N$ such that $$\ralg(\Delta)\leq \ralg(GL_d(\F_p))=d^2$$ and $$[\Gamma/N:\Delta]\leq \monSym(\ralg(GL_d(\F_p)))=\monSym(d^2).$$
Let $\Gamma^\circ$ be the preimage of $\Delta$ under the projection $\Gamma \to \Gamma/N$. we have
$$\ralg(\Gamma^\circ)=\ralg(\Delta)\leq d^2.$$
By the main theorem (Theorem \ref{thm:main}), $$\nu(\Gamma^\circ)\leq C(d^2).$$ 
Lemma \ref{lem:mult.iso}\eqref{lem:mult.iso:2} implies that 
$$\nu(\Gamma)=[\Gamma:\Gamma^\circ] \nu(\Gamma^\circ)=[\Gamma/N:\Delta] \nu(\Gamma^\circ) \leq \monSym(d^2)C(d^2)=C^{fin}(d).$$
\end{proof}

\begin{proof}[Proof of Corollary \ref{cor:main}] 
Set $\Lambda=C^{fin}(C^{lin}(d))$ where $C^{lin}$ is the function given by Lemma \ref{cor:faithful.rep.local}. 

Let $\psi:\mathbf{G} \rightarrow \mathbf{R}$ be the reductive quotient of $\mathbf{G}$ and let $\varphi : \mathbf{R}\to GL_{\co{C^{lin}}(d),F}$ be the embedding given by \co{Lemma} \ref{cor:faithful.rep.local}. The group $\varphi\circ\psi(K)$ is a compact subgroup  of $GL_{\co{C^{lin}}(d)}(F)$ and, after conjugation, we may assume that $\varphi \circ \psi(K)\subset GL_{\co{C^{lin}}(d)}(O_F)$. Let $K_1\subset GL_{\co{C^{lin}}(d)}(O_F)$ be the first congruence subgroup.

Set $M:=K\cap (\varphi\circ\psi)^{-1}(K_1)$, $L:=M\cap \ker \rho $, and $P:=L \cap \theta(L)$. The claim now follows by applying Theorem \ref{thm:main.str} to $\Gamma := K/P$, $N:=M/P$, and the embedding $\Gamma / N = K/M \hookrightarrow \GL_{\co{C^{lin}}(d)}(O_F)/K_1=\GL_{\co{C^{lin}}(d)}(\mathbb{F}_p)$.


\end{proof} 
\appendix
\section{Bounds on twisted multiplicities for spherical spaces of finite groups of Lie type} \label{sec:Shai}
\Rami{In this appendix we prove the following:}
\begin{theorem} \label{thm:Shai} Let $\cS$ be a scheme of finite type, let $\cG \rightarrow \cS$ be a connected reductive group scheme of finite type over $\cS$, and let $\calH \subseteq \cG$ be a closed (not necessarily connected) reductive subgroup scheme. Assume that, for every geometric point $s$ of $\cS$, the pair $(\cG_s,\calH_s)$ is spherical. Then there is a constant $C$ such that, for every finite field $F$, any $s\in \cS(F)$, any irreducible representation $\rho$ of $\cG_s(F)$, and any 1-dimensional character $\theta$ of $\calH_s(F)$,
\[
\dim\, \Hom_{\cG_s(F)} \left( \rho,\Ind_{\calH_s(F)}^{\cG_s(F)} \theta \right)  < C.
\]
\end{theorem}

The proof is very similar to the one in \cite{She}. The main additional ingredient is the geometrization of (1-dimensional) characters of finite groups of Lie type.

Given a Weil $\overline{\mathbb{Q}_\ell}$-local system $\fL$ on a scheme $X$ over a finite field $F$, the sheaf to function correspondence gives a function $X(F) \rightarrow \overline{\mathbb{Q}_\ell}$ which we denote by $\chi_\fL$.

\begin{lemma} \label{lem:geom.1.d.char} Let $\cS$ be a scheme of finite type and let $\calH \rightarrow \cS$ be a (not necessarily connected) reductive group scheme of finite type over $\cS$. There is a constant $C$ such that, for every finite field $F$ of size grater then $9$, any prime $\ell \neq char(F)$, any $s\in \cS(F)$, and any 1-dimensional character $\chi: \calH_s(F) \rightarrow \overline{\mathbb{Q}_\ell}$, there is a $\overline{\mathbb{Q}_\ell}$-local system $\fL$ over $\calH_{s}$ with a Weil structure of pure weight zero such that $\chi_\fL$ is a character of a representation of $\calH_s(F)$ of dimension at most $C$ that contains $\chi$ as one of its irreducible constituents.
\end{lemma} 

\begin{proof} For every (not necessarily connected) reductive group $\mathbf{G}$ defined over a field $F$, let $(\mathbf{G}^\circ)'$ be the derived subgroup of the connected component of $\mathbf{G}$, and let $\widetilde{\mathbf{G}}$ be the the product of the universal cover of $(\mathbf{G}^\circ)'$ and the radical of $\mathbf{G}$. The map $\widetilde{\mathbf{G}}(\bar F) \rightarrow \mathbf{G}(\bar F)$ has finite kernel and cokernel. 

By a stratification argument, there is a constant $c$ such that, for every field $F$, any $s\in \cS(F)$, the kernel and cokernel of the map $\widetilde{\calH_s}(\bar F) \rightarrow \calH_s(\bar F)$ are bounded by $c$. By Lemma \ref{lem:coker}, if $F$ is a finite field, the kernel and cokernel of the map $\widetilde{\calH_s}({ F} ) \rightarrow \calH_s({ F} )$ are bounded by $c^2$. We will show that the lemma holds with $C=c^4$.

Given $F ,\ell,s,\chi$ as in the lemma, denote the map $\widetilde{\calH_s} \rightarrow \calH_s$ by $\varphi$. Let $R(\calH_s)$ be the radical of $\calH_s$. By \cite[\S5]{Lus}, there is a 1-dimensional local system $\mathcal{F}$ on $R(\calH_s)$ such that $\chi_\mathcal{F}=\chi \restriction_{R(\calH_s)}$.

By Theorem \ref{thm:simp.con}\eqref{thm:simp.con:2}, the character $\chi\circ \varphi\restriction_{\widetilde{\calH_s'}(F)}$ is trivial. Therefore, under the decomposition $\widetilde{\calH_s}=\widetilde{\calH_s'}\times R(\calH_s)$, we have 
	$$\chi\circ \varphi=1_{\widetilde{\calH_s'}(\F_p)}\boxtimes \chi\restriction_{R(\calH_s)}= 1_{\widetilde{\calH_s'}(\F_p)}\boxtimes \chi_\mathcal{F}=\chi_{\mathbb{1}_{\widetilde{\calH_s'}} \boxtimes \mathcal{F}},$$
	where $\mathbb{1}_{\widetilde{\calH_s'}}$ denotes the 1-dimensional trivial local system on $\widetilde{\calH_s'}$

Setting $\fM=\varphi_!(\mathbb{1}_{\widetilde{\calH_s'}}\boxtimes \mathcal{F})$, the function $\chi_\fM$ is equal to $\varphi_*(\chi \circ \varphi)$. It follows that the restrictions of $\chi_\fM$ and $|\ker \varphi (F)|\chi$ to $\varphi(\widetilde{ \calH_s}(F))$ coincide. Choose coset representatives $g_1,\ldots,g_c\in\calH_s(F)$ to $\varphi(\widetilde{\calH_s}(F))$ and let $\fL = \bigoplus \Ad(g_i)^*(\fM)$. We have that 
\[
\chi_\fL=\Ind_{\varphi(\widetilde{\calH_s}(F))}^{\calH_s(F)} \chi_\fM \restriction_{\varphi(\widetilde{ \calH_s}(F))}=\Ind_{\varphi(\widetilde{\calH_s}(F))^{\calH_s(F)}} \Res_{\varphi(\widetilde{\calH_s}(F))}^{\calH_s(F)} |\ker \varphi (F)|\chi,
\]
which implies the claim.
\end{proof} 

Now we will continue with the original argument of \cite{AA_bnd,She}, adapting it to include the character $\chi$ and its geometrization $\fL$.

\begin{definition} If $\bfG$ is an algebraic group acting on a variety $\bfX$, we let $\co{\bfX}_\co{\bfG} =\left\{ (x,g)\in X \times \co{\bfG}  \mid g \cdot x=x \right\}$.
\end{definition} 

\begin{lemma}\label{lem:loc.sys.pull} Let $\co{\bfG} ,\co{\bfH}$ be reductive algebraic groups, let $\co{\bfX}=\co{\bfG} /\co{\bfH}$, and consider the diagram
\[
\xymatrix{\co{\bfG}  \times \co{\bfH} \ar[d]^{p} \ar[r]^a & \co{\bfX}_\co{\bfG} \\ \co{\bfH} & }
\]
where $a(g,h)=(g\co{\bfH},g h g ^{-1})$ and $p(g,h)=h$. If $\co{\fL}$ is an $\mathbf{H}$-equivariant local system on $\co{\bfH}$, then there is a local system $\co{\fM}$ on $X_G$ such that $a^* \co{\fM} \cong p^* \co{\fL}$.
\end{lemma} 

\begin{proof} \co{We construct $\mathfrak{M}$ using descent.} The main point is that $(\co{\bfG}  \times \co{\bfH}) \times_{\co{\bfX}_\co{\bfG}} (\co{\bfG}  \times \co{\bfH})\cong\co{\bfG}  \times \co{\bfH} \times \co{\bfH}$ via the map $((g_1,h_1),(g_2,h_2)) \mapsto (g_1,g_2 ^{-1} g_1,h_1)$. We get a diagram
\[
\xymatrix{\co{\bfG}  \times \co{\bfH} \times \co{\bfH} \ar@/^2.0pt/[r]^{a_1} \ar@/_2.0pt/[r]_{a_2} \ar[d]^{p_2} & \co{\bfG}  \times \co{\bfH} \ar[r]^a \ar[d]^{p_1} & \co{\bfX}_\co{\bfG} \\ \co{\bfH} \times \co{\bfH} \ar@/^2.0pt/[r]^{b_1} \ar@/_2.0pt/[r]_{b_2} & \co{\bfH} & }
\]
where $a_1,a_2,p_1,p_2$ are the projections, $b_1(x,y)=x$, and $b_2(x,y)=x y x ^{-1}$. We have $b_ip_2=p_1a_i$ for $i=1,2$. The equivariance of $\co{\fL}$ gives an identification $\alpha : b_1^* \co{\fL} \rightarrow b_2^* \co{\fL}$, which gives an identification $\beta:a_1^* p_1^* \co{\fL} = (b_1p_2)^*\co{\fL} \rightarrow (b_2p_2)^* \co{\fL} = a_2^* p_1^* \co{\fL}$, and it is easy to check that $\beta$ satisfies the cocycle identity. By descent, $p_1^* \co{\fL}$ is the pullback of a sheaf $\co{\fM}$ on $\co{\bfX}_\co{\bfG}$. Since $a$ is onto, we get that $\co{\fM}$ is a local system.
\end{proof} 

In the next lemma, we use the notion of induced character sheaf, see \cite[Definition 2.2.1]{AA_bnd}.

\begin{lemma} \label{lem:geom.mult.bound} Let $\co{\cS}$ be a scheme of finite type, $\co{\cG} \rightarrow \co{\cS}$ be a connected reductive group scheme of finite type over $\co{\cS}$, and let $\co{\calH} \subseteq \co{\cG}$ be a closed (not necessarily connected) reductive subgroup scheme. Assume that, for every geometric point $s$ of $\co{\cS}$, the pair $(\co{\cG}_s,\co{\calH}_s)$ is spherical. There is a constant $C_1$ such that, for any finite field $\co{F} $, any $s\in \co{\cS}(\co{F} )$, any induced character sheaf $\co{\fK}$ on $\co{\cG}_s$ and any $\co{\calH}_s$-equivariant local system $\co{\fL}$ on $\co{\calH}_s$ of weight zero, $\dim \Hom_{\co{\cG}_s(\co{F} )}\left( \chi_{\co{\fK}},\Ind_{\co{\calH}_s(\co{F} )}^{\co{\cG}_s(\co{F} )} \chi_\co{\fL} \right) < C_1 \cdot \rk \co{\fL}$.
\end{lemma} 

\begin{proof} For every geometric point $s$ of $\co{\cS}$, let $\Fl_s$ be the flag variety of $\co{\cG}_s$. Let $C_1$ be such that, for every geometric point $s$ of $\co{\cS}$, the number of connected components of $(\co{\cG}_s/ \co{\calH}_s \times \Fl_s)_{\co{\cG}_s}$ is bounded by $C_1$.

Let $\co{F} ,s,\co{\fK},\co{\fL}$ be as in the lemma. Denote $\co{\bfG}=\co{\cG}_s, \co{\bfH}=\co{\calH}_s,\co{\bfX}=\bfG/\bfH$, and fix a Borel subgroup $\bfB$ of $\bfG$ defined over $F$. Consider the diagram
\[
\xymatrix{& (\bfX \times \bfG/\bfB)_\bfG \ar[dr]^{\widetilde{f}} \ar[dd]_{q} \ar[dl]_{\widetilde{\pi}} & \\ \bfX_\bfG \ar[dr]_f & & (\bfG/\bfB)_\bfG \ar[dl]^\pi \\ & \bfG \ar[d]_p& \\ & \Spec(F) &}
\]
where $f,\pi,\widetilde{f},\widetilde{\pi}$ are the projections and $q=\pi\circ \widetilde{f}=\widetilde{\pi}\circ f$.

By definition, there is a weight zero local system $\mathfrak{F}$ on $(\bfG/\bfB)_\bfG$ such that $\co{\fK}$ is a direct summand of $R \pi_*\mathfrak{F}$. Applying Lemma \ref{lem:loc.sys.pull} to $\co{\fL}$, we get a local system $\co{\fM}$ on $\bfX_\bfG$. Since $\co{\fL}$ has weight zero, so does $\co{\fM}$. 

By construction, $\chi_{Rf_! \co{\fM}}=\Ind \chi_\co{\fL}$. Denoting the standard inner product of functions on $G(\co{F} )$ by $\langle - , - \rangle$, we get

\begin{align*}
\dim \Hom_{\bfG(\co{F} )} \left( \chi_\co{\fK} , \Ind_{\bfH(\co{F} )}^{\bfG(\co{F} )} \chi_\co{\fL} \right) 
&= \left \langle \chi_\co{\fK} , \Ind_{\bfH(\co{F} )}^{\bfG(\co{F} )} \chi_\co{\fL} \right \rangle \\
&= \left \langle \chi_{\co{\fK}}, \chi_{f_! \co{\fM}} \right \rangle \\
&= \left \langle \chi_{\co{\fK} \otimes f_! \co{\fM} ^\vee}, 1\right \rangle = \trace \left( Fr_F\mid p_! \co{(\fK \otimes f_! \fM^\vee)} \right),
\end{align*}
\co{where $Fr_F$ is induced by the Frobenius map of $\Spec(F)$.}

For every $n$, denote the degree $n$ extension of $\co{F} $ by $\co{F} _n$ and denote the pullbacks of $\co{\fL},\co{\fK}$ to $\bfG_{\co{F}_n},\bfH_{\co{F} _n}$ by $\co{\fL}_n,\co{\fK}_n$. From the same reasoning as before, we get
\[
\dim \Hom_{\bfG(\co{F} _n)} \left( \chi_{\co{\fK}_n} , \Ind_{\bfH(\co{F} _n)}^{\bfG(\co{F} _n)} \chi_{\co{\fL}_n} \right) = \trace \left( Fr_F ^n \mid p_! \co{(\fK \otimes f_! \fM^\vee)} \right).
\]

The complex $p_! \left( \co{\fK} \otimes f_! \co{\fM} ^\vee \right)$ is a direct summand of $p_! \left( \pi_! \mathfrak{F} \otimes f_! \co{\fM}^\vee \right)$. Since
\[
\pi_! \widetilde{f}_! \left( \widetilde{f}^* \mathfrak{F} \otimes \widetilde{\pi}^* \co{\fM}^\vee \right) = \pi_! \left( \mathfrak{F} \otimes \widetilde{f}_! \widetilde{\pi}^* \co{\fM}^\vee \right) = \pi_! \left( \mathfrak{F} \otimes \pi^* f_! \co{\fM}^\vee \right) = \pi_! \mathfrak{F} \otimes f_! \co{\fM}^\vee,
\]
we get that $p_! \left( \co{\fK} \otimes f_! \co{\fM} ^\vee \right)$ is a direct summand of $p_! q_! \left( \widetilde{f}^* \mathfrak{F} \otimes \widetilde{\pi}^* \co{\fM}^\vee \right)$. Since $\widetilde{f}^* \mathfrak{F} \otimes \widetilde{\pi}^* \co{\fM}^\vee$ has weight zero, we get that the complex $p_! \left( \co{\fK} \otimes f_! \co{\fM} ^\vee \right)$ has weight zero, is concentrated in degrees $0,\ldots,2\dim (\bfX \times \bfG/\bfB)_G=2\dim \bfG$, and 
\[
\dim H^{2\dim \bfG}p_! \left( \co{\fK} \otimes f_! \co{\fM} ^\vee \right) \leq \dim H^{2\dim \bfG}p_! q_! \left( \widetilde{f}^* \mathfrak{F} \otimes \widetilde{\pi}^* \co{\fM}^\vee \right) \leq c \rk \co{\fL}.
\]
Thus,
\[
\limsup_{n \rightarrow \infty} \trace \left( Fr_F^n \mid p_! \co{\fK} \otimes f_! \co{\fM} ^\vee \right) \leq c \rk \co{\fL} 
\]

so, by \cite[Lemma 2.4.1 and the proof of Theorem 2.1.3]{AA_bnd},
\[
\dim \Hom_{\bfG(\co{F} )} \left( \chi_\co{\fK} , \Ind_{\bfH(\co{F} )}^{\bfG(\co{F} )} \chi_\co{\fL} \right) = \trace \left( Fr_F \mid p_! \co{\fK} \otimes f_! \co{\fM} ^\vee \right) \leq c \rk \co{\fL}
\]

\end{proof} 

\begin{proposition} \label{prop:LL} For any $d$, there are integers $N,C$ such that, if 
\begin{enumerate}
\item $F$ is a finite field of size greater than $N$,
\item $\bfG$ is a reductive group  defined over $F$ of dimension at most $d$,
\item $\chi$ is an irreducible character of $\bfG(F)$,
\end{enumerate}
then there are induced character sheaves $\fK_1,\ldots,\fK_C$ and real numbers $\alpha_1,\ldots,\alpha_C\in[-C,C]$ such that $\sum \alpha_i \chi_{\fK_i}-\chi$ is a non-negative combination of irreducible characters of $\bfG(F)$.
\end{proposition}

\begin{proof} \cite[Lemma A.1, Theorem 2.2, Theorem 3.3]{She} (the last two are due to Laumon and Lusztig).
\end{proof}

\begin{proof}[Proof of Theorem \ref{thm:Shai}] 
Without loss of generality we may assume that $|F|>9$. Let $\rho$ be an irreducible representation of $G(F)$ and let $\theta$ be a character of $H(F)$. Let $C,\fK_i,\alpha_i$ be as in Lemma \ref{prop:LL} (applied to $\rho$), and let $\fL$ be as in Lemma \ref{lem:geom.1.d.char} (applied to $\theta$). By Lemma \ref{lem:geom.mult.bound}, we have
\[
\dim \Hom \left( \rho, \Ind \theta \right) \leq \left\langle \sum \alpha_i \chi_{\fK_i} , \Ind \chi_\fL \right\rangle \leq \sum | \alpha_i | \cdot \left| \left\langle \chi_{\fK_i} , \Ind \chi_\fL \right\rangle\right| < C^2 \cdot C_1 \cdot C
\]
\end{proof}

\section{A versal family of symmetric pairs of reductive groups over finite fields}\label{app:fam} In this appendix we prove Lemma \ref{cor:faithful.rep.local} and construct a family of symmetric pairs of reductive groups that includes all symmetric pairs of reductive groups of a given dimension over all finite fields (Lemma \ref{lem:fam.of.sym}).

\subsection{Proof of Lemma \ref{cor:faithful.rep.local}} 
For the proof we will need the following:
\begin{lemma} \label{lem:split.ext} There is an increasing function $C^{spt}: \mathbb{N} \rightarrow \mathbb{N}$ such that any reductive algebraic group $\mathbf{G}$ over an arbitrary  field $F$ splits over an extension $F'/F$ of degree at most $C^{spt}(\dim\, \mathbf{G})$.
\end{lemma} 

\begin{proof} 
Set $C^{spt}(d):=8^d(d+1)!3^{d^3}$.
	
There is a maximal torus of $\mathbf{G}$ that is defined over $F$ (\cite[XIV 1.1]{SGA3}), so we can assume that $\mathbf{G}$ is a torus. Denote $d=\dim\, \mathbf{G}$. 
	
Since $\bG$ is an $F$-form of $\mathbb{G}_m^{d}$, we get a continuous homomorphism $\rho :\Gal_F \rightarrow \Aut(\mathbb{G}_m^{d})=\GL_d(\mathbb{Z})$. The image of $\rho$ is finite, so, by Jordan's theorem (see, for example \cite{Col}), there is an abelian subgroup $A \subset \rho(\Gal_F)$ of index at most $8^d(d+1)!$.
	
If $g\in A$, then the order of $g$ is finite, so its minimal polynomial is product of cyclotomic polynomials $\Phi_{k_i}$. We get that $\phi(k_i):=\deg(\Phi_{k_i})\leq d$. This implies that $k_i\leq 3^{log_2(d)+1}$, and, thus, $g$ is semisimple of order at most  $3^{d(log_2(d)+1)}< 3^{d^2}.$
	
	Since $A$ is simultaneously diagonalizable (over $\mathbb{C}$), we get that $$|\rho(\Gal_F)| \leq 8^d(d+1)!3^{d^3}=C^{spt}(d).$$
\end{proof} 


\begin{proof}[Proof of Lemma \ref{cor:faithful.rep.local}] Since there are finitely many split reductive groups of given dimension, there is a function $C^{linSpt}$ such that every split reductive group $\mathbf{H}$ has a faithful representation of dimension $C^{linSpt}(\dim\,\mathbf{H})$. Given an arbitrary reductive group $\mathbf{G}$, Lemma \ref{lem:split.ext} implies that $\mathbf{G}$ splits over an extension $F'/F$ of degree at most $C^{spt}(\dim\, \mathbf{G})$. Hence, there is a faithful representation $\mathbf{G} \rightarrow \Res_{F'/F}\GL_{C^{linSpt}(\dim\, \mathbf{G})}$, so we can take $C^{lin}(n)=C^{linSpt}(n)C^{spt}(n)$.
\end{proof}

\subsection{Sketch of the proof of Lemma \ref{lem:fam.of.sym}} 
We first show that there are finitely many root data of a given dimension (See Lemma \ref{lem:f.root.dat} below). Thus, we restrict our attention to a given root datum $\fX$. We denote by $\cG$ the split reductive group scheme corresponding to $\fX$.
By Lemma \ref{lem:split.ext} there is an integer $k$ such that any reductive group of type  $\fX$ over a finite field splits after passing to a field extension of degree $k$.

We then construct a finite etale map of schemes $\cE\to \cF$ that forms a family containing all degree $k$ extensions of finite fields. This means that, for any degree $k$ extension of finite fields $E/F$, we can find an $F$-point $y$ of  $\cF$ whose fiber $\cE_y$ is $\spec E$.  Moreover, we equip $\mathcal{E}$ with an action of the cyclic group $C_k$ such that, if $F$ is a finite field, we can find $y$ as above such that the action of $C_k$ on $\mathcal{E}_y$ is the Frobenius. See Lemma \ref{lem:f.F.and.Fr.fam} below. 

By Lang's theorem, a reductive group of type $\fX$ over a finite field $F$ that splits over a degree $k$ extension $E/F$ is determined by an action of $C_k$ on $\fX$. We show that there are finitely may such actions up to conjugation (see Lemma \ref{lem:f.act.on.root.dat} below). Thus, we can fix one such an action $\xi$. We can also consider $\xi$ as an action on $\cG$.

At this point we can construct a group scheme $\calH \to \cF$ containing all groups of type $(\frak{X},\xi)$ over finite fields. Namely, we first construct a group scheme $\calH'\to \cF$ whose fiber over $y\in \mathcal{F}(F)$ is the restriction of scalars of $\cG_{\cE_y}$ to $F$. Using the two actions of $C_k$ on $\mathcal{E}$ and $\mathcal{G}$, we equip $\mathcal{H}'$ with an action of $C_k$. Finally, set $\calH := (\calH' )^{C_k}$.

Next, we incorporate all possible involutions. We first note that, up to inner automorphisms, there are only finitely many involutions of $\fX$ commuting with the action of $C_k$ (see Lemma \ref{lem:f.act.on.root.dat} below). Thus, we can restrict our attention to a specific such involution $\eta$. We then construct an $\mathcal{F}$-scheme $\cS$ whose $F$-points are pairs $(y,t)$ consisting of a point $y\in \cF(F)$  and an involution $t$ of $\cG_{\cE_y}$ which commutes with $\xi$ and is of outer class $\eta$.

Finally, we pull back the group schemes $\calH'$ and $\calH$ to $\cS$ and denote the resulting groups schemes $\cR'$ and $\cR$. Both $\mathcal{R}$ and $\mathcal{R}'$ are equipped with a natural involution $\tau$. The group scheme $\cR\to \cS$ with the involution $\tau$ gives the required family.

\begin{remark} In the proof below we skip $\mathcal{H},\mathcal{H}'$ and construct $\mathcal{R},\mathcal{R}'$ directly.
\end{remark} 

\subsection{Some preparations} 
	
\begin{lemma}\label{lem:f.root.dat} For any integer $n>0$  there is a finite number of isomorphism classes of complex connected reductive groups of dimension $n$.
\end{lemma}
\begin{proof} Fix a complex connected reductive group $\bG$. Let $\tilde \bG'$ be  the universal cover of its derived group, let $Z^0(\bG)$ be the connected component of the center of $\bG$, and let $\Gamma$ be the kernel of the multiplication map $\tilde \bG' \times Z^0(\bG) \to \bG$. 

Let $Z(\tilde \bG')$ be the center of  $\tilde \bG'$. Note that $Z(\tilde{\mathbf{G}}')$ is finite, that $\Gamma \subset Z(\tilde \bG') \times Z^0(\bG)$, and that $\Gamma \cap Z^0(\bG)$ is trivial. Thus, $\Gamma$  is a graph of a morphism from  subgroup of $Z(\tilde \bG')$ to $Z^0(\bG)$. This implies that $\Gamma< Z(\tilde \bG') \times Z^0(\bG)[|Z(\tilde \bG')|]$, where for an integer $k$ the group $Z^0(\bG)[k]$ is the subgroup of elements of order dividing $k$ in $Z^0(\bG)$.

Any complex connected reductive group $\bG$ is uniquely determined (up to isomorphism)  by the following:
\begin{itemize}
\item the simply conected semi-simple complex group $\tilde \bG'$.
\item the complex algebraic torus $Z^0(\bG)$.
\item the finite subgroup  $\Gamma< Z(\tilde \bG') \times Z^0(\bG)[|Z(\tilde \bG')|]$.
\end{itemize}
Since each of those has only finitely many options given the dimension of $\bG$, the claim follows.
\end{proof}

\begin{lemma}\label{lem:f.act.on.root.dat} For any complex connected reductive group $\mathbf{G}$ and any finite abelian group $A$,
$$\# Mor(A,Out(\mathbf{G}))/Ad(Out(\mathbf{G}))<\infty$$
\end{lemma}

\begin{proof} Any automorphism of $\mathbf{G}$ is determined by its restrictions to the derived subgroup $\bG'$ and to the connected component $Z^0(\bG)$ of the center. We first claim that the map $\Aut(\mathbf{G})\rightarrow \Aut(Z^0(\mathbf{G}))\times Aut(\bG')$ has finite cokernel. Indeed, let $\bfK$ be the kernel of the map $\mathbf{G}' \times Z^0(\mathbf{G}) \rightarrow \mathbf{G}$. The group $\mathbf{K}$ is finite. Let $\bfM \subset \mathbf{G}' \times Z^0(\mathbf{G})$ be the product of the center $Z(\mathbf{G}')$ and the finite group of elements of $Z^0(\mathbf{G})$ of order dividing $|\bfK|$. The group $\mathbf{M}$ is finite, contains $\mathbf{K}$, and is characteristic in $\mathbf{G}' \times Z^0(\mathbf{G})$. It follows that the subgroup of $\Aut(Z^0(\mathbf{G}))\times Aut(\bG')$ fixing $\mathbf{K}$ has finite index. Any element in this subgroup extends to an automorphism of $\mathbf{G}$.

Let $\phi$ be the composition $Aut(\mathbf{G})\rightarrow Aut(Z^0(\mathbf{G}))\times Aut(\bG') \rightarrow Aut(Z^0(\mathbf{G}))\times Out(\bG')$. By the paragraph above, the cokernel of $\phi$ is finite. Note also that the kernel of $\phi$ is the subgroup of inner automorphisms. In particular, we have an embedding $Out(\mathbf{G}) \rightarrow Aut(Z^0(\mathbf{G})) \times Out(\mathbf{G}')$ with a finite cokernel.

The group $Out(\bG')$ is finite. Denote it by $\Gamma$. The group $Aut(Z^0(\mathbf{G}))$ is isomorphic to 
$GL_n(\Z)$ for some integer $n$.
We get:
\begin{align*}
|Mor(A,Out(\mathbf{G}))/Ad(Out(\mathbf{G}))|&\leq |Mor(A,GL_n(\Z)\times \Gamma)/Ad(Out(\mathbf{G}))|
	\\&\leq[GL_n(\Z)\times \Gamma:Out(\mathbf{G})] \cdot |Mor(A,GL_n(\Z)\times \Gamma)/Ad(GL_n(\Z)\times \Gamma)|
	\\&\leq [GL_n(\Z)\times \Gamma:Out(\mathbf{G})]\cdot |\Gamma| \cdot |Mor(A,GL_n(\Z))/Ad(GL_n(\Z))|.
\end{align*}
It remains to prove that $Mor(A,GL_n(\Z))/Ad(GL_n(\Z))$ is finite. We have a map  $$\phi:Mor(A,GL_n(\Z))/Ad(GL_n(\Z))\to Mor(A,GL_n(\Q))/Ad(GL_n(\Q)).$$
Basic representation theory of finite groups implies that $Mor(A,GL_n(\Q))/Ad(GL_n(\Q))$ is finite, so it remains to prove that the fibers of $\phi$ are finite. This follows from the fact that the fibers of 
$$\psi:Mat_n(\Z)^{|A|}/Ad(GL_n(\Z))\to Mat_n(\Q)^{|A|}/Ad(GL_n(\Q))$$ are finite, which is  a special case of 
	\cite[Theorm 4.9]{PR}
\end{proof}


\subsection{Construction of the family} 

\begin{lemma} \label{lem:f.F.and.Fr.fam}
For any integer $n>0$, there exists a finite etale morphism  $\Psi_n:\cE_n\to \cF_n$ of schemes of finite type over $\Z$ with an action of $C_n$ on 
$\cE_n$ over $\cF_n$
such that, for any degree $n$ extension $E/F$ of finite fields, there exists $\nu:\spec F \to \cF_n$  such that 
$$\spec(E)  \simeq \spec(F)\times_{\cF_n} \cE_n$$ as a $C_n$-scheme. Here the action of $C_n$ on $E$ is the Galois action.
\end{lemma}
\begin{proof}
For a unital ring $A$ and an integer $k$, let $A_k[t]$ be the set of polynomials of degree $\leq k$ and let $A'_k[t]$ be the set of monic polynomials of degree $k$. Denote the resultant of two polynomials $f(t),g(t)$ by $res_t(f,g)$. Let 
	$$\cF_n(A):=\{(f,g)\in A'_n[t] \times A_{n-1}[t] \mid res_t(f,f')\in A^\times \text{ and } f\text{ divides }f\circ g\text{ and }g^{\circ n}-t\}$$ and let
	$$\cE_n(A):=\{(f,g,z)\in A'_n[t]\times A_{n-1}[t]\times A\mid (f,g)\in \cF_n \text{ and } f(z)=0 \}.$$
Define an action of $C_n$ on $\cE_n(A)$ by $$k \cdot (f,g,z)\mapsto(f,g,g^{\circ k}(z)).$$
		
By construction, the assignments $A\mapsto \cF_n(A)$ and $A\mapsto \cE_n(A)$ give rise to representable functors. We denote the representing schemes by $\cF_n$ and $\cE_n$. Similarly, the action of $C_n$ on $\cE_n(A)$ gives rise to an action of $C_n$ on $\cE_n$ over  $\cF_n$. Denote by $\Psi_n:\cE_n \to \cF_n$ the projection. The map $\Psi_n$ is an etale map.

Suppose that $E/F$ is a degree $n$ extension of finite fields. Let $\alpha \in E$ be a generator, and let $f\in F_n[t]$ be its (monic) minimal polynomial. Let $g\in F_{n-1}[t]$ be the polynomial of degree $<n$ such that $Fr_F(\alpha)=g(\alpha)$. The tuple $(f,g)$ is a point in $\cF_n(F)$, i.e., it defines a morphism  $\spec F\to \cF_n$. It is easy to see that 
		$$\spec(E)  \simeq \spec(F)\times_{\cF_n} \cE_n,$$
 as required.
\end{proof}

We now prove the main result of this appendix.
\begin{proof}[\Rami{Proof of Lemma \ref{lem:fam.of.sym}}]
Let $\fX$ be a pair consisting of a root datum and a choice of positive roots, and let $k$ be an integer. Let  $\alpha:S_2\times C_k\to Aut(\fX)$ be a morphism.
	
Let $\cG_\fX\to \spec \Z$ be the split reductive group scheme corresponding to $\fX$. Let $\alpha_2:S_2\times C_k\to Aut(\cG_\fX)$ be the corresponding action.

	Let $Aut_{\cG_{\fX}/\Z}:Schemes^{op} \to Groups$ be the functor defined by $Aut_{\cG_{\fX}/\Z}(S)=Aut_{S}(\cG_{\fX}\times_{\spec\Z}S)$, cf. \cite[Definition 7.1.3]{Con}. By \cite[Theorem 7.1.9]{Con}, this functor is representable by a (not necessarily finite type) $\Z$-group scheme that we also denote $Aut_{\cG_{\fX}/\Z}$.
	


Denoting $S_2=\left\{ 1,\eps \right\}$, let $\cI_{\fX,\alpha}\subset Aut_{\cG_{\fX}/\Z}$ be defined by

\[
\cI_{\fX,\alpha}(S):=\left\{ a \in Aut_{\cG_{\fX}/\Z}(S) \left\vert \text{ \parbox{10cm}{$a$ commutes with $\alpha_2(C_k)_S$ and, for every geometric point $s$ of $S$, the automorphism $a_s$ is in the class $\alpha(\eps,0)$}} \right. \right\}, 
\]
where $S$ is an affine scheme. For an automorphism $\beta$ of $\cG_{\fX}$, we denote by $\beta_S$ its restriction to $\cG_{\fX} \times S$. By \cite[Theorem 7.1.9]{Con}, $\cI_{\fX,\alpha}$ is of finite type. 
	
	Define an action $\alpha_3:S_2\times C_k\to Aut_{\cI_{\fX,\alpha}}( \cG_{\fX} \times_{} \cI_{\fX,\alpha})$ by
	$$\alpha_3(\eps^i j)(x,\eta)=(\eta ^i \alpha_2(j)x,\eta),$$ where $i\in \Z, j\in C_k$ and $(x,\eta)\in \cG_{\fX} \times_{} \cI_{\fX,\alpha}(S)$.
	
Let $\mathcal{F}_k,\mathcal{E}_k$ be as in Lemma \ref{lem:f.F.and.Fr.fam} and define $\cS_{\fX,\alpha}:=(\mathcal{I}_{\frak{X},\alpha} \times \mathcal{F}_k)\, ^{\wedge}_{\mathcal{F}_k} \mathcal{E}_k$, where $^{\wedge}$ denotes the internal morphism space, see e.g. \cite[\S\S3.1]{AA_sta}. An $F$-point of $\mathcal{S}_{\frak{X},\alpha}$ is a pair $(z,y)$, where $y\in \mathcal{F}_k(F)$ and $z\in \mathcal{I}_{\frak{X},\alpha}((\mathcal{E}_k)_y)$. Let  
\[
\mathcal{R}'_{\frak{X},\alpha}=\left( \mathcal{G}_\frak{X} \times \mathcal{S}_{\frak{X},\alpha} \right) ^{\,\,\wedge} _{\mathcal{S}_{\frak{X},\alpha}} \left( \mathcal{E}_k \times_{\mathcal{F}_k} \mathcal{S}_{\frak{X},\alpha} \right).
\]
Note that  $\cR'_{\fX,\alpha}$ has a natural structure of a group scheme over $\cS_{\fX,\alpha}$. By their constructions, $\mathcal{E}_k,\mathcal{S}_{\frak{X},\alpha},\mathcal{R}'_{\frak{X},\alpha}$ all have an action of $S_2 \times C_k$ ($S_2$ acts trivially on $\mathcal{E}_k$). Denoting the $S_2 \times C_k$-action on $\mathcal{R}'_{\mathfrak{X},\alpha}$ by $$\alpha_4:S_2\times C_k\to Aut_{\cS_{\fX,\alpha}}(\cR'_{\fX,\alpha}),$$ let 
	$$\cR_{\fX,\alpha}:=(\cR'_{\fX,\alpha})^{\alpha_4(C_k)},$$ and $$t_{\fX,\alpha}:=\alpha_4(\eps)|_{\cR_{\fX,\alpha}}.$$
Denote $n_\fX:=C^{spt}(\dim_{\Spec \mathbb{Z}}(\cG_\fX)),$ where $C^{spt}$ is the function given by Lemma \ref{lem:split.ext}.

Let 
$$\Delta_n:=\{(\fX,d,\kappa)|\fX \text{ is a root datum of dimension } \leq n; d \leq n_\fX ; \kappa \in Mor(S_2\times C_d,\Aut(\fX))/ad(\Aut(\fX))\}.$$
	since there are finitely many root data of a given dimension (Lemma  \ref{lem:f.root.dat}) and finitely many actions of $S_2\times C_d$ ($d \leq n_{\fX}$) on a given root datum (Lemma \ref{lem:f.act.on.root.dat}), the set  $\Delta_n$ is finite.
Finally, set 
	$$\cS_{n}:=\bigsqcup_{(\fX,d,[\alpha]) \in \Delta_n}
	\cS_{\fX,\alpha},$$
	$$\cR_{n}:=\bigsqcup_{(\fX,d,[\alpha]) \in \Delta_n}
	\cR_{\fX,\alpha},$$
	and
	$$t_{n}:=\bigsqcup_{(\fX,d,[\alpha]) \in \Delta_n}
	t_{\fX,\alpha}.$$
We claim that $(\cR_{n},\cS_{n},t_n)$ satisfies the requirements of the lemma.

Parts (\ref{lem:fam.of.sym:1},\ref{lem:fam.of.sym:3}) follow from the fact that, for any geometric point $x$ of $\cS_{\fX,\alpha}$, the group scheme $(\cR_{\fX,\alpha})_x$ is reductive and its absolute root system is $\fX$. It remains to show Part \eqref{lem:fam.of.sym:2}.

Let $n$ be an integer, let $F$ be a finite field, let $\bG$ be a reductive group of dimension $\leq n$ defined over $F$, and let $t$ be an involution of $\mathbf{G}$. We need to find an element $w\in \cS_n(F)$ such that
	 $$(\bG,t)\simeq ((\cR_n)|_w,t_n|_{(\cR_n)|_w}).$$	
Let $\fX$ be the absolute root datum of $\bG$. By Lemma \ref{lem:split.ext}, there is a field extension $E/F$ of degree $d \leq n_\frak{X}$ and an isomorphism $\mathbf{G}_E \simeq (\mathcal{G}_\frak{X})_E$.

Denoting the group of $E$-automorphisms of the algebraic group $(\mathcal{G}_{\frak{X}})_{E}$ by $Aut_{E} \left( (\mathcal{G}_{\frak{X}})_{E} \right)$, we get an element in $H^1(Gal(E/F),Aut_{E}((\cG_\fX)_{E}))$. By Lang's theorem, this element comes from an element $H^1(Gal(E/F),Out_{E} ((\cG_\fX)_{E}))$ via the embedding $Out_{E}( (\cG_\fX)_{E})\cong Out(\fX)\subset Aut_{E} (\cG_\fX)_{E}$. Since the action of $\Gal(E/F)$ on $Out_{E}\left( (\cG_\fX)_{E} \right)$ is trivial, this element is a homomorphism $\xi: Gal(E/F)\to Out_{E} \left( (\cG_{\fX})_{E} \right)=Aut(\fX)$.

Let $[t]\in Aut(\fX)$ be the involution corresponding to $t\in \Aut(\bG)$. We get an action $\alpha:S_2\times C_d\to Aut(\fX)$. By Lemma \ref{lem:f.F.and.Fr.fam}, there is an element $y\in \cF_d(F)$ such that $(\cE_d)|_y=\Spec E$ and the action $C_d$ on this fiber is the Frobenius action.

Let  $t_{E}$ be the automorphism of  $\bG_{E}$ corresponding to $t$. We will consider it as an element in $Aut_{\cG_X/\Z}(E)=Aut_E\left( (\mathcal{G}_\frak{X})_E \right)$. By construction, $t\in\cI_{\fX,\alpha}(F)$. The tuple $(y,t)$ gives a point $w\in \cS_{\fX,\alpha}(F) \subset \cS_n(F)$. Finally,
		 $$(\bG,t)\simeq ((\cR_n)|_w,t_n|_{(\cR_n)|_w}).$$	

	
	
\end{proof}


\bibliographystyle{alpha}
\bibliography{Ramibib}

\end{document}